\documentclass[a4paper,11pt]{article}
\usepackage{amssymb,amsmath,amsfonts,amsthm}
\newtheorem{theorem}{Theorem}[section]

\newtheorem{lemma}[theorem]{Lemma}
\newtheorem{proposition}[theorem]{Proposition}
\newtheorem{definition}[theorem]{Definition}

\newcommand{\ds}{\displaystyle}

\newcommand{\NN}{\mathbb N}

\newcommand{\CC}{\mathbb C}
\newcommand{\RR}{\mathbb R}
\newcommand{\ZZ}{\mathbb Z}

\newcommand{\DD}{\mathcal D}
\newcommand{\SSS}{\mathcal S}

\allowdisplaybreaks

\newcommand{\beq}{\begin{eqnarray}}
\newcommand{\eeq}{\end{eqnarray}}

\newcommand{\beqs}{\begin{eqnarray*}}
\newcommand{\eeqs}{\end{eqnarray*}}
\begin{document}
\catcode`\@=11


  \renewcommand{\theequation}{\thesection.\arabic{equation}}
  \renewcommand{\section}%
  {\setcounter{equation}{0}\@startsection {section}{1}{\z@}{-3.5ex plus -1ex
   minus -.2ex}{2.3ex plus .2ex}{\Large\bf}}
\title{\bf Semilinear pseudodifferential equations in spaces of tempered ultradistributions }
\author{Marco Cappiello $^{\textrm{a *}}$, Stevan Pilipovi\'c $^{\textrm{b}}$ and Bojan Prangoski $^{\textrm{c}}$}
\date{}
\maketitle
\noindent
\def\thefootnote{}
\footnote{ 
\textit{keywords: Tempered ultradistributions, pseudodifferential operators, semilinear equations} 
\\ * Corresponding author:  marco.cappiello@unito.it}
\begin{abstract}
We study a class of semilinear elliptic equations on spaces of tempered ultradistributions of Beurling and Roumieu type. Assuming that the linear part of the equation is an
elliptic pseudodifferential operator of infinite order with a sub-exponential growth of its symbol and that the non linear part is given by an infinite sum of powers of $u$ with sub-exponential growth with respect to $u,$ we prove a regularity result in the functional setting of the quoted ultradistribution spaces for a weak Sobolev
type solution $u$.
\end{abstract}
\section{Introduction}
In this paper we consider a class of semilinear equations and prove a result of regularity in the spaces of tempered ultradistributions of Beurling and Roumieu type. These ultradistributions can be regarded as a global counterpart on $\RR^d$   of the local ultradistributions studied by Komatsu \cite{Komatsu1, Komatsu2, Komatsu3} and they represent a natural generalisation of non-quasi-analytic Gelfand-Shilov type ultradistributions, cf. \cite{GS, PilipovicU, PilipovicT}. As well as the Gelfand-Shilov spaces, they are also a good functional setting for global pseudodifferential operators of infinite order, namely with symbol $a(x,\xi)$ admitting sub-exponential growth in both $x$ and $\xi$, see \cite{C1,C2, BojanS}. Here we want to apply the pseudodifferential operators introduced by the third author in \cite{BojanS} to the study of semilinear equations of the form

\begin{equation}
\label{maineq}
Au=f+F[u]
\end{equation} where $A=a(x,D)$, $f$ is a given test function in our setting and $F[u]$ is a nonlinear term given by a suitable infinite sum of powers of $u$. In \cite{PP} we  investigated the class of operators of \cite{BojanS} in the context of the Weyl and the Anti Wick calculus  while in the recent paper \cite{CPP}, we considered the case of linear equations and proved a result of hypoellipticity via the construction of a parametrix. To treat semilinear equations, we need to adopt a more sophisticated method based on suitable commutators and nonlinear estimates. This method, previously used in \cite{BG} \cite{CGR1}-\cite{CN2}, \cite{GR} in the case of symbols corresponding to differential operators and of nonlinear terms of finite order, is used also here but with a more advanced technique since we consider infinite series both in the case of symbols and in the case of nonlinear terms. Examples of our elliptic symbols are $a(x,\xi)=e^{c\langle(x,\xi)\rangle^{1/m}}, m>1, c\in \mathbb R$, whereas for what concerns the nonlinear terms we can consider $F[u]=\sum_{\beta}c_{\beta}P_{\beta}(x)u^{|\beta|}$ where $P_{\beta}(x)$ are ultrapolynomials of the form $\sum_{\gamma}\tilde{c}_{\gamma}x^{\gamma}/\gamma!^m, m>1$ and $c_{\beta}$ are suitable complex numbers tending rapidly to zero. Actually, we will consider nonlinear terms also with the additional assumption that $u\in H^s(\mathbb R^d), s>d/2$. With this we have plenty of examples, for example $F[u]=P(x)\cos u$ or $F[u]=P(x)e^{u^k}, k\in \mathbb Z_+$, where $P$ has sub-exponential growth of the order related to the order of the growth of the symbol. In this way we can analyse  elliptic operators $A$ of infinite order and sub-exponential growth  as well as nonlinear terms with sub-exponential growth, not considered in the literature, which shows an  intrinsic connection of the pseudodifferential calculus of \cite{BojanS} with the spaces of ultradistributions.

  \par

The paper is organised as follows. In the next Section \ref{sec0} we introduce the main tools involved in the paper and state the main result, namely Theorem \ref{mainthm}. Section \ref{examples} contains some examples of elliptic operators and nonlinear terms which motivate our analysis and for which Theorem \ref{mainthm} holds. In Section \ref{sec1} we refine some results about the pseudodifferential operators studied in \cite{BojanS} and  we prove some
precise estimates for the norms of some composed operators which will be instrumental in the proof of Theorem \ref{mainthm}. Finally, Section \ref{sec2} is devoted to the proof of the theorem which will be divided in two parts, one corresponding to the proof of the decay properties of the solution and the other related to its regularity.

\section{Notation and the main theorem}\label{sec0}

Before stating our results, let us fix some notation and introduce the functional setting where they are obtained. In the sequel, the sets of integer, non-negative integer, positive integer, real and complex numbers are denoted as standard by $\ZZ$, $\NN$, $\ZZ_+$, $\RR$, $\CC$. We denote $\langle x\rangle =(1+|x|^2)^{1/2} $ for $x\in \RR^d$, $D^{\alpha}= D_1^{\alpha_1}\ldots D_d^{\alpha_d},\quad D_j^
{\alpha_j}={i^{-1}}\partial^{\alpha_j}/{\partial x}^{\alpha_j}$, $\alpha=(\alpha_1,\ldots,\alpha_d)\in\NN^d$. Fixed $B>0$, we shall denote by $Q_B^c$ the set of all $(x,\xi)\in \RR^{2d}$ for which we have $\langle x \rangle \geq B$ or $\langle \xi \rangle \geq B.$ Finally, for $s \in \RR$, we shall denote by $H^s(\RR^d)$ the Sobolev space of all $u \in \mathcal{S}'(\RR^d)$ for which $\langle \xi \rangle^{s}\hat{u}(\xi) \in L^2(\RR^d)$, where $\hat{u}$ denotes the Fourier transform of $u$.
Following \cite{Komatsu1}, in the sequel we shall consider sequences $M_{p}$ of positive numbers such that $M_0=M_1=1$ and satisfying
all or some of the following conditions:
 $\;\;(M.1)$ $M_{p}^{2} \leq M_{p-1} M_{p+1}, \; \; p \in\ZZ_+$;
$\;\;(M.2)$ $\ds M_{p} \leq c_0H^{p} \min_{0\leq q\leq p} \{M_{p-q} M_{q}\}$, $p,q\in \NN$, for some $c_0,H\geq1$;
 $\;\;(M.3)$  $\ds\sum^{\infty}_{p=q+1}{M_{p-1}}/{M_{p}}\leq c_0q {M_{q}}/{M_{q+1}}$, $q\in \ZZ_+$;
$\;\;(M.4)$ $\ds \left({M_p}/{p!}\right)^2\leq {M_{p-1}}/{(p-1)!}\cdot {M_{p+1}}/{(p+1)!}$, for all $p\in\ZZ_+$.
In some assertions in the sequel we could replace $(M.3)$ by the weaker assumption:
$\;\;(M.3)'$ $\ds \sum_{p=1}^{\infty}{M_{p-1}}/{M_p}<\infty$ (cf. \cite{Komatsu1}). We observe moreover that  $(M.4)$ implies $(M.1)$.
As an example of sequence satisfying all the conditions above we can take $M_p=p!^s$, $s>1$.
 For a multi-index $\alpha\in\NN^d$, $M_{\alpha}$ will mean $M_{|\alpha|}$, $|\alpha|=\alpha_1+...+\alpha_d$. We can associate to any sequence $M_p$ as above the function
$
M(\rho)=\sup_{p\in\NN}\log_+   {\rho^{p}}/{M_{p}} , \; \; \rho > 0.
$
This is a non-negative, continuous, monotonically increasing function which vanishes for sufficiently small $\rho>0$ and increases more rapidly than $\ln \rho^p$ when $\rho$ tends to infinity, for any $p\in\NN$ (cf. \cite{Komatsu1}). \par As in \cite{Komatsu1}, see also \cite{PilipovicT}, we shall denote by $\mathfrak{R}$ the set of positive sequences which monotonically increase to infinity. For $(r_p)\in\mathfrak{R}$, consider the sequence $N_0=1$, $N_p=M_pR_p$, $p\in\ZZ_+$, where we denote $R_p=\prod_{j=1}^{p}r_j$ (in the future we will often use this notation). It is easy to verify that this sequence satisfies $(M.1)$ and $(M.3)'$. Its associated function will be denoted by $N_{r_p}(\rho)$, i.e. $\ds N_{r_{p}}(\rho )=\sup_{p\in\NN} \log_+ {\rho^{p }}/(M_pR_p)$, $\rho > 0$. Note, for given $(r_{p})$ and every $k > 0 $ there is $\rho _{0} > 0$ such that $\ds N_{r_{p}} (\rho ) \leq M(k \rho )$ for $\rho > \rho _{0}$.\par
Now we can introduce the space of tempered ultradistributions and its test function space.
For $m >0$ and a sequence $M_p$ satisfying the conditions $(M.1)-(M.3)$, we shall denote by $\mathcal{S}_{\infty}^{M_p,m}(\RR^d)$ the Banach space of all functions $\varphi \in \mathcal{C}^{\infty}(\RR^d)$ such that
\begin{equation} \label{norm}
\|\varphi\|_m:=\sup_{\alpha\in \NN^d}\sup_{x \in \RR^d}\frac{m^{|\alpha|}| D^{\alpha}\varphi(x)| e^{M(m|x|)}}{M_{\alpha }}<\infty,
\end{equation}
endowed with the norm in \eqref{norm} and we denote $\ds\SSS^{(M_p)}(\RR^d)=\lim_{\substack{\longleftarrow\\ m\rightarrow\infty}} \SSS^{M_p,m}_{\infty}(\RR^d)$ and $\ds\SSS^{\{M_p\}}(\RR^d)=\lim_{\substack{\longrightarrow\\ m\rightarrow 0}} \SSS^{M_p,m}_{\infty}(\RR^d)$. In the sequel we shall consider simultaneously the two latter spaces by using the common notation $\mathcal{S}^{\ast}(\RR^d)$.
For each space we will consider a suitable symbol class. Definitions and statements will be formulated first for the $(M_p)$ case and then for the $\{M_p\}$ case, using the notation $\ast$. We shall denote by $\mathcal{S}^{\ast \prime}(\RR^d)$ the strong dual space of $\mathcal{S}^{\ast}(\RR^d)$. We refer to \cite{PilipovicK, PilipovicU, PilipovicT, BojanL} for the properties of $\mathcal{S}^{\ast}(\RR^d)$ and $\mathcal{S}^{\ast \prime}(\RR^d)$. Here we just recall that the Fourier transformation is an automorphism on $\mathcal{S}^{\ast}(\RR^d)$ and on $\mathcal{S}^{\ast \prime}(\RR^d)$ and that for $M_p=p!^s,\, s>1$, we have $M(\rho) \sim \rho^{1/s}$. In this case $\mathcal{S}^{\ast}(\RR^d)$ coincides respectively with the Gelfand-Shilov spaces $\Sigma_s(\RR^d)$ (resp. $\mathcal{S}_s(\RR^d)$) of all functions  $\varphi \in \mathcal{C}^{\infty}(\RR^d)$ such that
$$\sup_{\alpha, \beta \in \NN^d} h^{-|\alpha|-|\beta|}(\alpha!\beta!)^{-s} \sup_{x \in \RR^d}|x^\beta \partial^\alpha \varphi(x)|<\infty$$
for every $h >0$ (resp. for some $h >0$), cf. \cite{GS, PilipovicU}. A measurable function $f$ on $\RR^d$ is said to be of ultrapolynomial growth of class * if $\|f(\cdot)e^{-M(m|\cdot|)}\|_{L^{\infty}(\RR^d)}<\infty$ for some $m>0$ (resp. for every $m >0$).

\par
Following \cite{BojanS} we now introduce the class of pseudodifferential operators involved in the sequel.
Let $M_p, A_p$ be two sequences of positive numbers. We assume that $M_p$ satisfies  $(M.1)$, $(M.2)$, $(M.3)$ and $(M.4)$ and that $A_p$  satisfies
 $A_0=A_1=1$, $(M.1), (M.2), (M.3)'$ and $(M.4).$ Moreover we suppose that $A_p \subset M_p$ i.e. there exist $c_0 >0, L>0$ such that $A_p \leq c_0 L^p M_p$ for all $p \in \NN.$ Let $\rho_0=\inf\{\rho\in\RR_+|\, A_p\subset M_p^{\rho}\}$. Obviously $0<\rho_0\leq 1$. Let $\rho\in\RR_+$ be arbitrary but fixed such that $\rho_0\leq \rho\leq 1$ if the infimum can be attained, or otherwise $\rho_0<\rho\leq 1$. For any fixed $h>0,m >0$ we denote by $\Gamma_{A_p, \rho}^{M_p, \infty}(\RR^{2d};h,m)$ the space of all functions $a(x,\xi) \in \mathcal{C}^{\infty}(\RR^{2d})$ such that
$$ \sup_{\alpha, \beta \in \ZZ_+^d} \sup_{(x,\xi) \in \RR^{2d}}\frac{|D_{\xi}^\alpha D_x^\beta a (x,\xi)| \langle (x,\xi)\rangle^{\rho|\alpha+\beta|} e^{-M(m|x|)-M(m|\xi|)}}{h^{|\alpha+\beta|}A_{\alpha}A_{\beta}}<\infty,$$
where $M(\cdot)$ is the associated function for the sequence $M_p$. Then we define
$$
\Gamma_{A_p, \rho}^{(M_p), \infty}(\RR^{2d}; m)= \lim_{\stackrel{\longleftarrow}{h \rightarrow 0}}\Gamma_{A_p, \rho}^{M_p, \infty}(\RR^{2d};h,m);  \quad
\Gamma_{A_p, \rho}^{(M_p), \infty}(\RR^{2d})=\lim_{\stackrel{\longrightarrow}{m \to \infty}}\Gamma_{A_p, \rho}^{(M_p),\infty}(\RR^{2d};m);$$
$$\Gamma_{A_p, \rho}^{\{M_p\}, \infty}(\RR^{2d}; h)=\lim_{\stackrel{\longleftarrow}{m \rightarrow 0}}\Gamma_{A_p, \rho}^{M_p,\infty}(\RR^{2d};h,m); \quad
\Gamma_{A_p, \rho}^{\{M_p\}, \infty}(\RR^{2d})=\lim_{\stackrel{\longrightarrow}{h \rightarrow \infty}}\Gamma_{A_p, \rho}^{\{M_p\},\infty}(\RR^{2d};h).
$$
We  associate to any symbol $a \in \Gamma^{\ast, \infty}_{A_p, \rho}(\RR^{2d})$ a pseudodifferential operator $a(x,D)$
acting continuously on $\mathcal{S}^{\ast}(\RR^d)$ and on $\mathcal{S}^{\ast \prime}(\RR^d)$. A symbolic calculus for $\Gamma^{\ast, \infty}_{A_p, \rho}(\RR^{2d})$ (denoted there by $\Gamma^{\ast, \infty}_{A_p, A_p, \rho}(\RR^{2d})$) has been constructed in \cite{BojanS}. As a consequence it was proved that the class of pseudodifferential operators with symbols in $\Gamma^{\ast, \infty}_{A_p, \rho}(\RR^{2d})$ is closed with respect to composition and adjoints, cf. \cite{BojanS} and the next section for details. Moreover, in \cite{CPP} we consider hypoelliptic symbols in $\Gamma^{*,\infty}_{A_p, \rho}(\RR^{2d})$ and we proved the existence of parametrices for the associated operators. Now we need to introduce a notion of elliptic symbol in $\Gamma^{*,\infty}_{A_p, \rho}(\RR^{2d})$. For this purpose let $\tilde{M}_p$ be another sequence such that $\tilde{M}_0=\tilde{M}_1=1$ and satisfying $(M.1)$, $(M.2)$, $(M.3)'$ and $(M.4)$ and $M_p\subset \tilde{M}_p$, i.e. there exists $\tilde{c},\tilde{L}>0$ such that $M_p\leq \tilde{c}\tilde{L}^p\tilde{M}_p$ (observe that $\tilde{M}_p$ can be the same as $M_p$). Obviously, without losing generality, we can assume that the constant $H$ from $(M.2)$ is the same for the sequences $A_p$, $M_p$ and $\tilde{M}_p$. For $(k_p)\in\mathfrak{R}$ we denote by $\tilde{N}_{k_p}(\cdot)$ the associated function to the sequence $\tilde{M}_p \prod_{j=1}^p k_j$. One easily obtains the following inequalities
\beq\label{nrstv}
\tilde{M}(\lambda/\tilde{L})\leq M(\lambda)+\ln_+ \tilde{c} \mbox{ and } \tilde{N}_{k_p}(\lambda/\tilde{L})\leq N_{k_p}(\lambda)+\ln_+ \tilde{c},\,\, \forall\lambda>0.
\eeq
If $(k_p)\in\mathfrak{R}$ satisfies the condition $K_{p+q}\leq cH^{p+q} K_p K_q$ for some $c,H\geq 1$, where $K_p$ stands for $\prod_{j=1}^p k_j$ then the sequences $M_pK_p$ and $\tilde{M}_pK_p$ satisfy $(M.2)$ (since $M_p$ and $\tilde{M}_p$ do). Lemma 2.3 of \cite{BojanL} proves that there are plenty of sequences of this type (in fact, the quoted lemma claims that given a sequence of $\mathfrak{R}$ one can find such sequence which is smaller than the chosen one). Because of this property we will say that the sequence $(k_p)$ satisfies $(M.2)$ (although the precise statement will be to say that the sequence $K_p=\prod_{j=1}^p k_j$ satisfies $(M.2)$).
\begin{definition} \label{elliptic}
Given $A_p, M_p, \tilde{M}_p$ as before,
a symbol $a\in \Gamma^{*,\infty}_{A_p,\rho}\left(\RR^{2d}\right)$ is said to be $(\tilde{M}_p)$-elliptic, (resp. $\{\tilde{M}_p\}$-elliptic) if
\begin{itemize}
\item[$i)$] there exist $m,B,c>0$ (resp. there exist $(k_p)\in\mathfrak{R}$ which satisfies $(M.2)$ and $B,c>0$) such that
\beqs
|a(x,\xi)|\geq ce^{\tilde{M}(m|\xi|)}e^{\tilde{M}(m|x|)} (\mbox{ resp. } |a(x,\xi)|\geq ce^{\tilde{N}_{k_p}(|\xi|)}e^{\tilde{N}_{k_p}(|x|)}), \quad  (x,\xi) \in Q_B^c;
\eeqs
\item[$ii$)] for every $h>0$ there exists $C>0$ (resp. there exist $h,C>0$) such that
\beqs
\left|D^{\alpha}_{\xi}D^{\beta}_x a(x,\xi)\right|\leq C\frac{h^{|\alpha|+|\beta|}A_{\alpha+\beta}|a(x,\xi)|} {\langle(x,\xi)\rangle^{\rho(|\alpha|+|\beta|)}}, \quad (x,\xi) \in Q_B^c.
\eeqs
\end{itemize}
\end{definition}

Given $(k_p)\in\mathfrak{R}$ and $\gamma \in \NN^d$, from now on we will always use the notation $K_p$ for $\prod_{j=1}^p k_j$ and  $K_{\gamma}$ for $K_{|\gamma|}$ .\\
\indent Finally we introduce the class of nonlinear terms involved in our equations.

For $\beta \in \NN^d$, let $p_{\beta}(x)$ be smooth functions on $\RR^d$ such that for every $h>0$ there exists $C>0$ such that
\beq\label{343}
\left|D^{\alpha}_x p_{\beta}(x)\right|\leq C\frac{h^{|\alpha|+|\beta|}A_{\alpha} e^{\tilde{M}(h|x|)}}{\tilde{M}_{\alpha}} \mbox{ for all } \alpha,\beta\in\NN^d,
\eeq
(respectively
\beq\label{345}
\left|D^{\alpha}_x p_{\beta}(x)\right|\leq C\frac{h^{|\alpha|+|\beta|}A_{\alpha} e^{\tilde{N}_{k_p}(h|x|)}} {\tilde{M}_{\alpha}\prod_{j=1}^{|\alpha|}k_j} \mbox{ for all } \alpha,\beta\in\NN^d.)
\eeq
For such a family of functions $p_{\beta}(x)$ and $u \in H^s(\RR^d), s>d/2$ we can consider the function
\begin{equation}\label{nonlinearterm}F[u]=\ds\sum_{|\beta|=2}^{\infty}p_{\beta} u^{|\beta|},
\end{equation}
The condition $s>d/2$ implies that $F[u]$ is well defined and continuous on $\RR^d$ and $\left\|F[u] e^{-\tilde{M}(h|\cdot|)}\right\|_{L^{\infty}(\RR^d)}<\infty$ (resp. $\left\|F[u] e^{-\tilde{N}_{k_p}(h|\cdot|)}\right\|_{L^{\infty}(\RR^d)}<\infty$) for some $h$. This (together with (\ref{nrstv})) implies that $F[u]\in\SSS^{* \prime}(\RR^d)$.

The main result of the paper is the following one.

\begin{theorem} \label{mainthm}
Let $a \in \Gamma^{*,\infty}_{A_p,\rho}\left(\RR^{2d}\right)$ be $(\tilde{M}_p)$-elliptic (resp. $\{\tilde{M}_p\}$-elliptic) and let $f\in\SSS^*\left(\RR^d\right)$. Let $u\in H^s\left(\RR^d\right)$, $s>d/2$, be a solution of the equation \eqref{maineq} with $F[u]$ defined by \eqref{343} and \eqref{nonlinearterm} (resp. \eqref{345} and \eqref{nonlinearterm}). Then the following properties hold:
\\
\indent $i)$ For every $h>0$ there exists $C>0$ (resp. there exist $h,C>0$) such that $|u(x)|\leq C e^{-M(h|x|)}$. Moreover, $u\in\mathcal{C}^{\infty}(\RR^d)$ with the following estimate on its derivatives: there exists $\tilde{h}>0$ such that
\beqs
\sup_{\alpha}\frac{\tilde{h}^{|\alpha|}\|D^{\alpha} u\|_{L^{\infty}}}{\tilde{M}_{\alpha}}<\infty, \left(\mbox{ resp. } \sup_{\alpha}\frac{\tilde{h}^{|\alpha|}\|D^{\alpha} u\|_{L^{\infty}}}{\tilde{M}_{\alpha}\prod_{j=1}^{|\alpha|}k_j}<\infty\right).
\eeqs
\indent $ii)$ Furthermore, if $F[u]$ is a finite sum, then $u\in\SSS^*\left(\RR^d\right)$.
\end{theorem}

\section{Examples}
\label{examples}
In this section we give an example of a non standard elliptic operator $a$, given by (\ref{e1}) below, and of nonlinear terms to which Theorem \ref{mainthm} can be applied when $M_p=p!^{l}$, $l>1$. This examples underline the novelties of this paper with respect to the results obtained in \cite{CGR1, CGR3, CGR4}.

\noindent \textbf{Example 1.} Let $A_p= p!^{v}$, with $1<v<l$ and $0<\rho<1$ is such that $v\leq l\rho$. Let $a_0:(0,\infty)\rightarrow (0,\infty)$ be given by $\ds a_0(\lambda)=\sum_{n=0}^{\infty}\frac{h^n \lambda^n}{n!^{l+l'}}$. Then
\beqs
a_0^{(k)}(\lambda)&=&\frac{1}{\lambda^{k}}\sum_{n\geq k}\frac{n!}{(n-k)!}\cdot\frac{h^n \lambda^n}{n!^{l+l'}}\\
&=&\frac{h^{k/(l+l')}}{\lambda^{k(l+l'-1)/(l+l')}}\sum_{n\geq k}\left(\frac{h^n \lambda^n}{n!^{l+l'}}\right)^{(l+l'-1)/(l+l')}\cdot \left(\frac{h^{n-k} \lambda^{n-k}}{(n-k)!^{l+l'}}\right)^{1/(l+l')}\\
&\leq&\frac{h^{k/(l+l')}}{\lambda^{k(l+l'-1)/(l+l')}}\left(\sum_{n\geq k}\frac{h^n \lambda^n}{n!^{l+l'}}\right)^{(l+l'-1)/(l+l')}\cdot\left(\sum_{n\geq k} \frac{h^{n-k} \lambda^{n-k}}{(n-k)!^{l+l'}}\right)^{1/(l+l')}\\
&\leq& \frac{h^{k/(l+l')}}{\lambda^{k(l+l'-1)/(l+l')}}a_0(\lambda),
\eeqs
where, in the first inequality we used Holder's inequality with $p=(l+l')/(l+l'-1)$ and $q=l+l'$. Define
\begin{equation}\label{e1}
a(w)=a_0(\langle w\rangle), \;w\in\RR^{2d}
\end{equation}
Then $a\in\mathcal{C}^{\infty}(\RR^{2d})$. To prove that $a \in \Gamma^{*,\infty}_{A_p,\rho}(\RR^{2d})$ and that it is elliptic, we need some preliminary results. First we recall the following multidimensional variant of the Fa\`a di Bruno formula (see \cite[Corollary 2.10]{faadib}).

\begin{proposition}(\cite{faadib})
Let $|\alpha|=n\geq 1$ and $h(x_1,...,x_d)=f(g(x_1,...,x_d))$ with $g\in \mathcal{C}^{n}$ in a neighbourhood of $x^0$ and $f\in\mathcal{C}^n$ in a neighbourhood of $y^0=g(x^0)$. Then
\beqs
\partial^{\alpha}h(x^0)=\sum_{r=1}^{n}f^{(r)}(y^0)\sum_{p(\alpha,r)}\alpha!\prod_{j=1}^n \frac{\left(\partial^{\alpha^{(j)}}g(x^0)\right)^{k_j}}{k_j! \left(\alpha^{(j)}!\right)^{k_j}},
\eeqs
where
\beqs
p(\alpha,r)&=&\Bigg\{\left(k_1,...,k_n; \alpha^{(1)},...,\alpha^{(n)}\right)\Big|\, \mbox{for some}\, 1\leq s\leq n, k_j=0 \mbox{ and } \alpha^{(j)}=0\\
&{}&\mbox{for } 1\leq j\leq n-s;\, k_j>0 \mbox{ for } n-s+1\leq j\leq n; \mbox{ and}\\
&{}&0\prec \alpha^{(n-s+1)}\prec...\prec \alpha^{(n)} \mbox{ are such that}\\
&{}&\sum_{j=1}^n k_j=r,\, \sum_{j=1}^n k_j\alpha^{(j)}=\alpha\Bigg\}.
\eeqs
\end{proposition}

The relation $\prec$ used in this proposition is defined in the following way (cf. \cite{faadib}). We say that $\beta\prec\alpha$ when one of the following holds:
\begin{itemize}
\item[$(i)$] $|\beta|<|\alpha|$;
\item[$(ii)$] $|\beta|=|\alpha|$ and $\beta_1<\alpha_1$;
\item[$(iii)$] $|\beta|=|\alpha|$, $\beta_1=\alpha_1,...,\beta_k=\alpha_k$ and $\beta_{k+1}<\alpha_{k+1}$ for some $1\leq k<d$.
\end{itemize}

Before we continue, we need the following technical result.

\begin{lemma}\label{1110001}
For $\beta\in\NN^{d}$, the following estimate holds:
\beqs
\sum_{r=1}^{|\beta|}\sum_{p(\beta,r)}\frac{r!}{k_1!\cdot ... \cdot k_{|\beta|}!}\leq 2^{|\beta|(d+1)}
\eeqs
\end{lemma}

\begin{proof} Let $f(\lambda)=\lambda^{|\beta|}$ and $g(x)=(-x_1)^{-1}\cdot...\cdot(-x_d)^{-1}$. We apply the Fa\`a di Bruno formula to the composition $g(f(x))$ at the point $x=(-1,...,-1)$:
\beqs
\prod_{\substack{j=1\\ \beta_j\neq 0}}^d \left(|\beta|(|\beta|+1)\cdot...\cdot(|\beta|+\beta_j-1)\right)&=&\partial^{\beta}(g\circ f)(-1)\\
&=& \sum_{r=1}^{|\beta|}\frac{|\beta|!}{(|\beta|-r)!} \sum_{p(\beta,r)}\beta!\prod_{j=1}^{|\beta|}\frac{1}{k_j!}.
\eeqs
Hence, we obtain
\beqs
\prod_{\substack{j=1\\ \beta_j\neq 0}}^d \frac{|\beta|(|\beta|+1)\cdot...\cdot(|\beta|+\beta_j-1)}{\beta_j!}&=&\sum_{r=1}^{|\beta|}{|\beta|\choose r} \sum_{p(\beta,r)}\frac{r!}{k_1!\cdot ... \cdot k_{|\beta|}!}\\
&\geq& \sum_{r=1}^{|\beta|}\sum_{p(\beta,r)}\frac{r!}{k_1!\cdot ... \cdot k_{|\beta|}!}.
\eeqs
From this, the desired inequality follows, since
\beqs
\frac{|\beta|(|\beta|+1)\cdot...\cdot(|\beta|+\beta_j-1)}{\beta_j!}=\frac{(|\beta|+\beta_j)!} {\beta_j!(|\beta|-1)!(|\beta|+\beta_j)}\leq {|\beta|+\beta_j\choose \beta_j}\leq 2^{|\beta|+\beta_j}.
\eeqs
\end{proof}

If we apply the Fa\`a di Bruno formula to the composition of $a_0$ and $w\mapsto\langle w\rangle$ and use the well known estimate $|\partial^{\alpha}\langle w\rangle|\leq C' 2^{|\alpha|+1}|\alpha|! \langle w\rangle^{1-|\alpha|}$, $\alpha\in\NN^{2d}$, $w\in\RR^{2d}$, we have
\beqs
\left|D^{\alpha}a(w)\right|&\leq& \sum_{r=1}^{|\alpha|}\left|(D^r a_0)(\langle w\rangle)\right| \sum_{p(\alpha,r)}\alpha!\prod_{j=1}^{|\alpha|}\frac{\left|D^{\alpha^{(j)}}\langle w\rangle\right|^{k_j}}{k_j! (\alpha^{(j)}!)^{k_j}}\\
&\leq&\sum_{r=1}^{|\alpha|}\frac{h^{r/(l+l')}a(w)}{\langle w\rangle^{r(l+l'-1)/(l+l')}} \sum_{p(\alpha,r)}\alpha!\prod_{j=1}^{|\alpha|}\frac{C'^{k_j} 2^{|\alpha^{(j)}|k_j+k_j}\left(\left|\alpha^{(j)}\right|!\right)^{k_j} \langle w\rangle^{k_j-|\alpha^{(j)}|k_j}}{k_j! (\alpha^{(j)}!)^{k_j}}\\
&\leq&\frac{2^{|\alpha|}C_0^{|\alpha|}a(w)}{\langle w\rangle^{|\alpha|(l+l'-1)/(l+l')}} \sum_{r=1}^{|\alpha|} \sum_{p(\alpha,r)}\alpha!\prod_{j=1}^{|\alpha|}\frac{\left(\left|\alpha^{(j)}\right|!\right)^{k_j}}{(\alpha^{(j)}!)^{k_j}} \prod_{j=1}^{|\alpha|}\frac{1}{k_j!},
\eeqs
where we denote by $C_0$ the quantity $2C'h^{1/(l+l')}$. One can easily prove that for $\beta^{(1)},...,\beta^{(n)}\in\NN^{2d}$,
\beqs
\frac{\left(\beta^{(1)}+...+\beta^{(n)}\right)!}{\beta^{(1)}!\cdot...\cdot \beta^{(n)}!}\leq \frac{\left|\beta^{(1)}+...+\beta^{(n)}\right|!}{\left|\beta^{(1)}\right|!\cdot...\cdot \left|\beta^{(n)}\right|!}.
\eeqs
In fact one easily verifies this inequality for $n=2$ and the general case follows by induction. We obtain
\beq\label{11223}
\alpha!\prod_{j=1}^{|\alpha|}\frac{\left(\left|\alpha^{(j)}\right|!\right)^{k_j}}{(\alpha^{(j)}!)^{k_j}} \leq |\alpha|!.
\eeq
Using this estimate, together with Lemma \ref{1110001}, we conclude that
\beqs
\left|D^{\alpha}a(w)\right|\leq\frac{\left(2^{2d+2}C_0\right)^{|\alpha|}|\alpha|!a(w)}{\langle w\rangle^{|\alpha|(l+l'-1)/(l+l')}}.
\eeqs
Hence, if we take $l'>0$ large enough such that $(l+l'-1)/(l+l')\geq \rho$, we obtain $a\in \Gamma^{*,\infty}_{A_p,\rho}(\RR^{2d})$ and the growth condition $ii)$ of Definition \ref{elliptic} is satisfied for $a$. By the definition of $a$ the lower bound $i)$ of Definition \ref{elliptic} trivially holds for $\tilde{M}_p=p!^{l+l'}$ and some $m>0$ in the $(M_p)$ case (resp. for $\tilde{M}_p=p!^{l+l'/2}$ and $k_p=p^{l'/2}$ in the $\{M_p\}$ case). Hence $a$ is elliptic.

Symbols given in the Introduction, correspond to $m=l+l'$ with suitable $ \rho$. Actually,  applying the Fa\`a di Bruno formula to the composition of $\lambda\mapsto \lambda^{1/m}$ and $w\mapsto \langle w\rangle$, by the same technique as above using (\ref{11223}), Lemma \ref{1110001} and the estimate $\left|\left(\lambda^{1/m}\right)^{(r)}\right|\leq r!\lambda^{-r+1/m}$, $\lambda>0$, one can prove $\ds \left|D^{\alpha}\left(\langle w\rangle^{1/m}\right)\right|\leq C^{|\alpha|}|\alpha|!\langle w\rangle^{-|\alpha|+1/m}$. Applying again the Fa\`a di Bruno formula to the composition of $\lambda\mapsto e^{\lambda}$ and $w\mapsto \langle w\rangle^{1/m}$ we obtain the estimate $$\left|D^{\alpha}\left(e^{\langle w\rangle^{1/m}}\right)\right|\leq C^{|\alpha|}|\alpha|!e^{\langle w\rangle^{1/m}} \langle w\rangle^{-|\alpha|(m-1)/m}$$ which implies that such symbols can be considered in Theorem \ref{mainthm}.\\

\noindent 
\textbf{Example 2.} An interesting nontrivial example of nonlinear term satisfying the assumptions of Theorem \ref{mainthm} can be given in the following way. Define $p_{\beta}$ by
\beqs
p_{\beta}(x)=\sum_{\alpha\in\NN^d}c_{\alpha,\beta}x^{\alpha+\beta}
\eeqs
where $c_{\alpha,\beta}$ satisfy the following condition: for every $h>0$ there exists $C>0$ such that $\left|c_{\alpha,\beta}\right|\leq C h^{|\alpha|+|\beta|}/\tilde{M}_{\alpha+\beta}$ (resp. $\left|c_{\alpha,\beta}\right|\leq C h^{|\alpha|+|\beta|}/(\tilde{M}_{\alpha+\beta}\prod_{j=1}^{|\alpha|+|\beta|}k_j)$). By \cite[Proposition 4.5.]{Komatsu1} and by simple calculation we obtain that $p_{\beta}(x)=R^{-1}_{\beta}P_{\beta}(x)$ satisfy the condition of Theorem \ref{mainthm}, where $R_{\beta}=\prod_{j=1}^{|\beta|}r_j$ for some $(r_p)\in\mathfrak{R}$ (for example $r_p=p^{\epsilon}$, $\epsilon>0$) and $P_{\beta}$ are entire functions which satisfy the estimate: for every $h>0$ there exists $C>0$ such that $\left|P_{\beta}(z)\right|\leq C e^{\tilde{M}(h|z|)}$, $z\in\CC^d$ (resp. $\left|P_{\beta}(z)\right|\leq C e^{\tilde{N}_{k_p}(h|z|)}$, $z\in\CC^d$), i.e. $P_{\beta}$ are ultrapolynomials of class $\{\tilde{M}_p\}$ (resp. $\{\tilde{N}_p\}$, with $\tilde{N}_p=\tilde{M}_p\prod_{j=1}^p k_j$).  Interesting examples of this form are $F[u]=P(x)\cos u$ and $F[u]=P(x)e^{u^k}$, $k\in\ZZ_+$, where $P(x)$ is an ultrapolynomial with growth condition as stated above (to verify that these are indeed examples of the stated form one only needs to write $\cos u$ and $e^{u^k}$ in power series in $u$). These examples fit well in our analysis since we presumed in Theorem \ref{mainthm}
that $u\in H^s(\mathbb R^d)$ for some $s>d/2$, which implies that $u$ is bounded over $\mathbb R^d$.

\section{Pseudodifferential operators on $\mathcal{S}^{*}(\RR^d), \mathcal{S}^{*\prime}(\RR^d)$} \label{sec1}
In this section we recall some results contained in \cite{BojanS} and concerning the calculus for pseudodifferential operators with symbols in  $\Gamma^{\ast, \infty}_{A_p, \rho}(\RR^{2d})$. For the purposes of this paper we need to modify slightly some statements with respect to \cite{BojanS}. The proof of these new assertions is completely analogous to the original ones and do not deserve to be repeated here.

\par
Now we recall the notion of asymptotic expansion for symbols in $\Gamma^{\ast, \infty}_{A_p, \rho}(\RR^{2d})$, cf. \cite[Definition 2]{BojanS}.

\begin{definition}
Let $M_p$ and $A_p$ be as in the definition of $\Gamma_{A_p, \rho}^{*, \infty}(\RR^{2d})$ and let $m_0=0, m_p=M_p/M_{p-1}, p \in \ZZ_+$. We denote by $FS^{\ast, \infty}_{A_p, \rho}(\RR^{2d})$ the space of all formal sums $\sum_{j \in \NN}a_j$
such that for some $B>0$, $a_j \in \mathcal{C}^{\infty}(\textrm{int} \, Q_{Bm_j}^c)$ and satisfy the following condition: there exists $m>0$ such that for every $h >0$ (resp. there exists $h >0$ such that for every $m>0$) we have
$$\sup_{j \in \NN} \sup_{\alpha, \beta \in \NN^d} \sup_{(x,\xi) \in Q^c_{Bm_j}} \frac{|D_{\xi}^{\alpha}D_{x}^{\beta}a_j(x,\xi)|\langle (x,\xi)\rangle^{\rho(|\alpha+\beta|+2j)}e^{-M(m|x|)-M(m|\xi|)}}{h^{|\alpha+\beta|+2j}A_{\alpha}A_\beta A_j^2}<\infty.$$
\end{definition}

Notice that any symbol $a \in \Gamma^{\ast, \infty}_{A_p, \rho}(\RR^{2d})$ can be regarded as an element $\sum\limits_{j \in \NN}a_j$ of $FS^{\ast, \infty}_{A_p, \rho}(\RR^{2d})$ with $a_0=a, a_j =0$ for $j \geq 1.$

\begin{definition}
\label{asexp}
A symbol $a \in \Gamma^{\ast, \infty}_{A_p, \rho}(\RR^{2d})$ is equivalent to $\sum_{j \in \NN}a_j \in FS^{\ast, \infty}_{A_p, \rho}(\RR^{2d})$ (we write $a \sim \sum_{j \in \NN}a_j $ in this case) if there exist $m,B >0$ such that for every $h >0$ (resp. there exist $h,B>0$ such that for every $m >0$) the following condition holds:
$$\sup_{N \in \ZZ_+} \sup_{\alpha, \beta \in \NN^d} \sup_{(x,\xi) \in Q^c_{Bm_N}} \frac{\Big|D_{\xi}^{\alpha}D_{x}^{\beta}\big(a(x,\xi) - \sum\limits_{j<N}a_j(x,\xi)\big)\Big|e^{-M(m|x|)-M(m|\xi|)}}{h^{|\alpha+\beta|+2N}A_{\alpha}A_\beta A_N^2 \langle (x,\xi)\rangle^{-\rho(|\alpha+\beta|+2N)}}<\infty.$$
\end{definition}

An operator $a(x,D)$ with symbol $a \sim 0$ is $*$-regularizing, namely it extends to a linear and continuous map from $\SSS^{* \prime}(\RR^d)$ to $\SSS^*(\RR^d)$, see \cite[Theorem 3]{BojanS}.
Moreover, for every sum $\sum_{j \in \NN}a_j \in FS^{\ast, \infty}_{A_p, \rho}(\RR^{2d})$ one can find a symbol $a \in \Gamma^{\ast, \infty}_{A_p, \rho}(\RR^{2d})$ such that $a \sim \sum_{j \in \NN}a_j,$ cf. \cite[Theorem 4]{BojanS}. Actually by the same argument one can prove the following more precise assertion.

\begin{proposition}
Let $g$ be a positive continuous function such that $g(w)$ and $1/g(w)$ have ultrapolynomial growth of class * and let $U$ be a subset of $FS_{A_p,\rho}^{*,\infty}\left(\RR^{2d}\right)$ for which there exists $B>0$ such that for every $h>0$ there exists $C>0$ (resp. there exist $B,h,C>0$) such that
\beqs
\sup_{j\in\NN}\sup_{\alpha,\beta}\sup_{(x,\xi)\in Q_{Bm_j}^c}\frac{\left|D^{\alpha}_{\xi}D^{\beta}_x a_j(x,\xi)\right|
\langle (x,\xi)\rangle^{\rho|\alpha|+\rho|\beta|+2j\rho}}
{h^{|\alpha|+|\beta|+2j}A_{\alpha}A_{\beta}A_j^2 g(x,\xi)}\leq C
\eeqs
for all $\sum_{j \in \NN} a_j\in U$. Then for every $\tilde{h}>0$ there exists $C>0$ (resp. there exist $\tilde{h},C>0$) such that the following condition holds: for every sum $\sum a_j \in U$ there exists a symbol $a \sim \sum_{j \in \NN}a_j$ satisfying the following estimate:
\beqs
\left|D^{\alpha}_{\xi} D^{\beta}_x a(x,\xi)\right|\leq C\frac{\tilde{h}^{|\alpha|+|\beta|}A_{\alpha} A_{\beta} g(x,\xi)} {\langle (x,\xi)\rangle^{\rho(|\alpha|+|\beta|)}}.
\eeqs
\end{proposition}

In \cite[Theorem 7]{BojanS} it was proved that the composition of two operators $b(x,D)$ and $a(x,D)$ with symbols in $\Gamma^{\ast, \infty}_{A_p, \rho}(\RR^{2d})$ is the sum of an operator $f_{a,b}(x,D)$ with symbol $f_{a,b} \in \Gamma^{\ast, \infty}_{A_p, \rho}(\RR^{2d})$, with $f_{a,b} \sim \sum_{j}f_{a,b,j}$, where
\beq \label{ascomp}
f_{a,b,j}(x,\xi)=\sum_{|\alpha|=j}\frac{1}{\alpha!}\partial^{\alpha}_{\xi} b(x,\xi) D^{\alpha}_x a(x,\xi),
\eeq
and of a *-regularizing operator $T_{a,b}.$ In fact, $f_{a,b}=\sum_j (1-\chi_j)f_{a,b,j}$ where $\chi_j$ is defined in the following way (cf. the proof of \cite[Theorem 4]{BojanS}). Take $\varphi,\psi\in\DD^{(A_p)}\left(\RR^{d}\right)$, in the $(M_p)$ case, resp. $\varphi,\psi\in\DD^{\{A_p\}}\left(\RR^{d}\right)$ in the $\{M_p\}$ case, such that $0\leq\varphi,\psi\leq 1$, $\varphi(x)=1$ when $\langle x\rangle\leq 2$, $\psi(\xi)=1$ when $\langle\xi\rangle\leq 2$ and $\varphi(x)=0$ when $\langle x\rangle\geq 3$, $\psi(\xi)=0$ when $\langle\xi\rangle\geq 3$. Then define $\chi(x,\xi)=\varphi(x)\psi(\xi)$, $\ds \chi_n(x,\xi)=\chi\left(\frac{x}{Rm_n},\frac{\xi}{Rm_n}\right)$ for $n\in\ZZ_+$ and $\chi_0(x,\xi)=0$, where $m_n=M_n/M_{n-1}$ and $R>0$ is large enough.

\begin{proposition}\label{composition}
Let $U_1$ and $U_2$ be bounded subsets of $\Gamma^{(M_p),\infty}_{A_p,\rho}\left(\RR^{2d},m'\right)$ (resp. $\Gamma^{\{M_p\},\infty}_{A_p,\rho}\left(\RR^{2d},h'\right)$), for some $m'>0$ (resp. for some $h'>0$). Then for every $a\in U_1$ and $b\in U_2$ we have $b(x,D)a(x,D)=f_{a,b}(x,D)+T_{a,b}$ where $f_{a,b}=\sum_j (1-\chi_j)f_{a,b,j}$ and $\chi_j$ are the cut-off functions defined above which can be chosen uniformly for $a \in U_1, b \in U_2$, and with $f_{a,b,j}$ given by \eqref{ascomp}.
Moreover, the family $T_{a,b}$ of *-regularizing operators is an equicontinuous subset of $\mathcal{L}_b\left(\SSS'^*\left(\RR^d\right),\SSS^*\left(\RR^d\right)\right)$.
\end{proposition}

\indent From the results above we notice that in general the composition of two operators with symbols in $\Gamma^{\ast, \infty}_{A_p, \rho}(\RR^{2d})$ is still an operator of infinite order.
In the sequel we will be interested to the case when the composition is a finite order operator with bounded symbol, hence the related operator is bounded on Sobolev spaces.
With this purpose we give the following definition.

\begin{definition}
Let $V,W\subseteq\Gamma^{*,\infty}_{A_p,\rho}\left(\RR^{2d}\right)$ and let $f(w)$ be a positive continuous function on $\RR^{2d}$ such that $f(w)$ and $1/f(w)$ are of ultrapolynomial growth of class * (see \cite{KomatsuI}). The sets $V$ and $W$ are said to be $(f,*)$-conjugate if for every $h>0$ there exists $C>0$ (resp. there exist $h,C>0$) such that
\beqs
\left|D^{\alpha} a(w)\right|\leq C\frac{h^{|\alpha|}A_{\alpha}}{\langle w\rangle^{\rho|\alpha|}f(w)} \mbox{ and } \left|D^{\alpha} b(w)\right|\leq C\frac{h^{|\alpha|}A_{\alpha}f(w)}{\langle w\rangle^{\rho|\alpha|}} \mbox{ for all } a\in V,\, b\in W.
\eeqs
\end{definition}

Obviously if $V$ and $W$ are $(f,*)$-conjugate then they are bounded subsets of $\Gamma^{(M_p),\infty}_{A_p,\rho}\left(\RR^{2d};m'\right)$ for some $m'>0$ (resp. $\Gamma^{\{M_p\},\infty}_{A_p,\rho}\left(\RR^{2d};h'\right)$, for some $h'>0$).

\begin{proposition}\label{330}
Let $V$ and $W$ be $(f,*)$-conjugate. Then, there exists $C>0$ such that
\beqs
\left\|b(x,D)a(x,D)\right\|_{\mathcal{L}_b(H^s)}\leq C, \mbox{ for all } a\in V,\, b\in W.
\eeqs
\end{proposition}

\begin{proof} Let $f_{a,b}$ be the symbol of the operator $b(x,D)a(x,D)$ defined as above. Then $b(x,D)a(x,D)=f_{a,b}(x,D)+T_{a,b}$, where $T_{a,b}$ form an equicontinuous subset of $\mathcal{L}_b\left(\SSS'^*\left(\RR^d\right),\SSS^*\left(\RR^d\right)\right)$. Then $f_{a,b} \sim \sum_j f_{a,b,j}$, where
\beqs
f_{a,b,j}(w)=\sum_{|\nu|=j}\frac{1}{\nu!}\partial^{\nu}_{\xi}b(w)D^{\nu}_x a(w).
\eeqs
Observe that
\beqs
\left|D^{\alpha}_w f_{a,b,j}(w)\right|&\leq& \sum_{\beta\leq\alpha}\sum_{|\nu|=j}\frac{1}{\nu!}\left|D^{\alpha-\beta}_wD^{\nu}_{\xi}b(w)\right| \left|D^{\beta}_wD^{\nu}_x a(w)\right|\\
&\leq& C_1\sum_{\beta\leq\alpha}\sum_{|\nu|=j}\frac{1}{\nu!}\frac{h^{|\alpha|+2|\nu|}A_{|\alpha|+2|\nu|}}{\langle w\rangle^{\rho(|\alpha|+2|\nu|)}}\leq C_2\frac{(2hH)^{|\alpha|+2j}A_{\alpha}A_{2j}}{\langle w\rangle^{\rho(|\alpha|+2j)}}.
\eeqs
Now, since $f_{a,b}=\sum_j (1-\chi_j)f_{a,b,j}$ with $\chi_j$ defined as above, one easily obtains that for every $h>0$ there exists $C>0$, resp. there exist $h,C>0$ such that $\left|D^{\alpha}_w f_{a,b}(w)\right|\leq C h^{|\alpha|}A_{\alpha}\langle w\rangle^{-\rho|\alpha|}$. From this it follows that $f_{a,b}(x,D)$, $a\in V$, $b\in W$, form a bounded subset of $\mathcal{L}_b(H^s)$ (cf. \cite[Theorem 1.7.14]{NR} and \cite[Theorem 2.1.11]{NR} and its proof), the claim follows.
\end{proof}

The next result has been proved in \cite{CPP} for more general hypoelliptic operators. It is immediate to verify that it holds in particular for symbols satisfying the ellipticity conditions in Definition \ref{elliptic}.

\begin{theorem} \label{parametrix}
Let $a\in\Gamma^{*,\infty}_{A_p,\rho}\left(\RR^{2d}\right)$ be $(\tilde{M}_p)$-elliptic (resp. $\{\tilde{M}_p\}$-elliptic). Then there exists a *-regularizing operator $T$ and a symbol $b\in\Gamma^{*,\infty}_{A_p,\rho}\left(\RR^{2d}\right)$ such that $b(x,D)a(x,D)=\mathrm{Id}+T$. Moreover, the symbol $b$ satisfies the following condition:
there exists $B'>0$ such that for every $h>0$ there exist $C>0$ (resp. there exist $h,C>0$) such that
\beq\label{grwww}
\left|D^{\alpha}_{\xi}D^{\beta}_x b(x,\xi)\right|\leq C \frac{h^{|\alpha|+|\beta|}A_{\alpha+\beta}}{|a(x,\xi)|\langle (x,\xi)\rangle^{\rho(|\alpha|+|\beta|)}},\qquad (x,\xi) \in Q_{B'}^c.
\eeq
\end{theorem}

As we mentioned in the introduction, given $(r_p)\in\mathfrak{R}$ and multi-index $\gamma$, $R_{\gamma}$ stands for $R_{|\gamma|}=\prod_{j=1}^{|\gamma|}r_j$.

\begin{lemma}\label{340}
Let $a \in \Gamma^{*,\infty}_{A_p, \rho}(\RR^{2d})$ be $(\tilde{M}_p)$-elliptic (resp. $\{\tilde{M}_p\}$-elliptic) and let $b$ be the symbol of the parametrix of $a(x,D)$. Then the sets $\{b\}$ and $\ds\left\{\frac{h^{|\alpha|}}{M_{\alpha}}D^{\alpha}_w a\Big|\, \alpha\in\NN^{2d}\right\}$ are $(|a(w)|,*)$-conjugate for every $h>0$ (resp. for some $h>0$). Hence, for every $h>0$ there exists $C>0$ (resp. there exist $h,C>0)$ such that the estimate $\left\|b(x,D)\circ\left(D^{\alpha}_w a\right) (x,D)\right\|_{\mathcal{L}_b(H^s)}\leq Ch^{|\alpha|}M_{\alpha}$ holds. Moreover, in the $(M_p)$ case, there exist $(r_p)\in\mathfrak{R}$ with $r_1=1$ and $C>0$ such that $\left\|b(x,D)\circ\left(D^{\alpha}_w a\right) (x,D)\right\|_{\mathcal{L}_b(H^s)}\leq CM_{\alpha}/R_{\alpha}$.
\end{lemma}

\begin{proof} We have
\beqs
\frac{h^{|\alpha|}}{M_{\alpha}}\left|D^{\alpha+\beta}_w a(w)\right|&\leq& C_1\frac{h^{|\alpha|}h_1^{|\alpha|+|\beta|}A_{\alpha+\beta}|a(w)|}{M_{\alpha}\langle w\rangle^{\rho(|\alpha|+|\beta|)}}\\
&\leq&
C_2\frac{(LHhh_1)^{|\alpha|}(Hh_1)^{|\beta|}A_{\beta}|a(w)|}{\langle
w\rangle^{\rho|\beta|}}. \eeqs The $\{M_p\}$ case trivially
follows from this by choosing $h$ small enough, since $h_1$ is
fixed. In the $(M_p)$ case, for each fixed $h$ we can take $h_1$
arbitrary small. One easily sees that this implies that the sets
under consideration are $(|a(w)|,(M_p))$-conjugate. The inequality
$\left\|b(x,D)\circ\left(D^{\alpha}_w a\right)
(x,D)\right\|_{\mathcal{L}_b(H^s)}\leq Ch^{|\alpha|}M_{\alpha}$
follows by Proposition \ref{330}. It remains to prove the last
part of the lemma. Since, in the $(M_p)$ case, we already proved
that $\ds\sup_{\alpha\in
\NN^{2d}}\frac{\left\|b(x,D)\circ\left(D^{\alpha}_w a\right)
(x,D)\right\|_{\mathcal{L}_b(H^s)}}
{h^{|\alpha|}M_{\alpha}}<\infty$ for every $h>0$, by Lemma 3.4 of \cite{Komatsu3}
we can conclude that there exists $(\tilde{r}_p)\in\mathfrak{R}$
and $C>0$ such that $\left\|b(x,D)\circ\left(D^{\alpha}_w a\right)
(x,D)\right\|_{\mathcal{L}_b(H^s)}\leq
CM_{\alpha}/\prod_{j=1}^{|\alpha|}\tilde{r}_j$. If we take
$r_p=\max\{\tilde{r}_p,1\}$, then $(r_p)\in\mathfrak{R}$, $r_1=1$
and the desired estimate holds for this $(r_p)$, possibly with
a larger constant $C$.
\end{proof}

\begin{lemma}\label{349}
There exists $l\geq 1$ such that the sets $\ds \left\{\frac{h^{|\beta|}}{l^{|\gamma|}\tilde{M}_{\gamma}}D^{\alpha}_x p_{\beta} \xi^{\gamma}\Bigg|\, \alpha,\beta,\gamma\in\NN^d\right\}$ and $\{b\}$ are $\left(e^{\tilde{M}(m|x|)}e^{\tilde{M}(m|\xi|)},(M_p)\right)$-conjugate for any $h>0$, (resp. the sets  $\ds \left\{\frac{h^{|\beta|}}{l^{|\gamma|}K_{\gamma}\tilde{M}_{\gamma}}D^{\alpha}_x p_{\beta} \xi^{\gamma}\Bigg|\, \alpha,\beta,\gamma\in\NN^d\right\}$ and $\{b\}$  are $\left(e^{\tilde{N}_{k_p}(|x|)}e^{\tilde{N}_{k_p}(|\xi|)},\{M_p\}\right)$-conjugate for any $h>0$). In particular for each $\tilde{h}>0$ there exists $C>0$ such that $\left\|(B\circ D^{\alpha}p_{\beta}(x) \partial^{\gamma}) (x,D)\right\|_{\mathcal{L}(H^s)}\leq C\tilde{h}^{|\beta|}l^{|\gamma|}\tilde{M}_{\gamma}$ (resp. $\left\|(B\circ D^{\alpha}p_{\beta}(x) \partial^{\gamma}) (x,D)\right\|_{\mathcal{L}(H^s)}$ $\leq C\tilde{h}^{|\beta|}l^{|\gamma|}K_{\gamma}\tilde{M}_{\gamma}$).
\end{lemma}

\begin{proof} We consider first the $(M_p)$ case. Pick $l\geq 1$ such that $H^2/l\leq m/12$. Let $h,h'>0$ be arbitrary but fixed. Pick $0<h_1<1$ such that $H\sqrt{h_1}\leq h'$, $L\tilde{L}Hh_1\leq 1$, $hh_1\leq 1$ and $H^2\sqrt{h_1}\leq m/6$. Then\\
\\
$\ds\frac{h^{|\beta|}}{l^{|\gamma|}\tilde{M}_{\gamma}}\left|D^{\mu}_{\xi}D^{\nu}_x\left(D^{\alpha}_x p_{\beta}(x) \xi^{\gamma}\right)\right|$
\beqs
&\leq& \frac{h^{|\beta|}\gamma!}{l^{|\gamma|}\tilde{M}_{\gamma}(\gamma-\mu)!}|\xi|^{|\gamma|-|\mu|}\left|D^{\alpha+\nu}_x p_{\beta}(x)\right|\\
&\leq& C_1 \frac{2^{|\gamma|}h^{|\beta|}h_1^{|\alpha|+|\beta|+|\nu|}A_{\alpha+\nu}\mu!\langle(x,\xi)\rangle^{|\mu|+|\nu|}} {l^{|\gamma|}\tilde{M}_{\gamma}\tilde{M}_{\alpha+\nu}\langle(x,\xi)\rangle^{\rho(|\mu|+|\nu|)}} |\xi|^{|\gamma|-|\mu|}e^{\tilde{M}(h_1|x|)}\\
&\leq& C_2 \frac{(2H/l)^{|\gamma|}(H^2h_1)^{|\nu|}(Hh_1)^{|\alpha|} (hh_1)^{|\beta|}A_{\alpha}A_{\nu}\mu!\langle(x,\xi)\rangle^{|\gamma|+|\nu|}} {\tilde{M}_{\gamma+\nu}\tilde{M}_{\alpha}\langle(x,\xi)\rangle^{\rho(|\mu|+|\nu|)}} e^{\tilde{M}(h_1|x|)}\\
&\leq& C_3 \frac{(H\sqrt{h_1})^{|\nu|}(L\tilde{L}Hh_1)^{|\alpha|}(hh_1)^{|\beta|}h'^{|\mu|}A_{\nu+\mu}} {\langle(x,\xi)\rangle^{\rho(|\mu|+|\nu|)}}e^{\tilde{M}(h_1|x|)}e^{\tilde{M}((H\sqrt{h_1}+2H/l)\langle(x,\xi)\rangle)}.
\eeqs
Since $\tilde{M}_p$ satisfies $(M.2)$, by Proposition 3.6 of \cite{Komatsu1}, we have
\beqs
e^{\tilde{M}(h_1|x|)}e^{\tilde{M}((H\sqrt{h_1}+2H/l)\langle(x,\xi)\rangle)}&\leq& C_4e^{\tilde{M}(h_1|x|)}e^{\tilde{M}(3(H\sqrt{h_1}+2H/l)|x|)}e^{\tilde{M}(3(H\sqrt{h_1}+2H/l)|\xi|)}\\
&\leq&
C_5e^{\tilde{M}(3(H\sqrt{h_1}+2H/l)|\xi|)}e^{\tilde{M}(3(H^2\sqrt{h_1}+2H^2/l)|x|)}.
\eeqs If we use this in the above estimate, by the way we defined
$h_1$, we have \beqs
\frac{h^{|\beta|}}{l^{|\gamma|}\tilde{M}_{\gamma}}\left|D^{\mu}_{\xi}D^{\nu}_x\left(D^{\alpha}_x
p_{\beta}(x) \xi^{\gamma}\right)\right|\leq
C\frac{h'^{|\mu|+|\nu|}A_{\mu+\nu}}{\langle(x,\xi)\rangle^{\rho(|\mu|+|\nu|)}}
e^{\tilde{M}(m|\xi|)}e^{\tilde{M}(m|x|)}, \eeqs which proofs the
$(M_p)$ case. In the $\{M_p\}$ case one can use the same technique
as above (observe that the sequence $K_p\tilde{M}_p$ satisfies
$(M.2)$). The last part follows by Proposition \ref{330}.
\end{proof}

\begin{lemma}\label{350}
Let $h>0$ and for each $\beta\in\NN^d$, let $p_{\beta}(x)$ be a smooth function satisfying \eqref{343} in the $(M_p)$ case (resp. satisfying \eqref{345} in the $\{M_p\}$ case) and let $j_{\beta} \in\{1,...,d\}$. Then the following properties hold:\\
\indent $a)$ The sets $\{b\}$ and $\ds\left\{h^{|\beta|}x_{j_{\beta}} p_{\beta}(x)\big|\, \beta\in\NN^d\right\}$ are $\left(e^{\tilde{M}(m|x|)}e^{\tilde{M}(m|\xi|)},(M_p)\right)$-conju\-gate (resp. $\left(e^{\tilde{N}_{k_p}(|x|)}e^{\tilde{N}_{k_p}(|\xi|)},\{M_p\}\right)$-conjugate). In particular, for every $\tilde{h}>0$ there exists $C>0$, such that $\left\|(B\circ x_{j_{\beta}} p_{\beta}(x)) (x,D)\right\|_{\mathcal{L}_b(H^s)}\leq C\tilde{h}^{|\beta|}$. Moreover, there exist $(r_p)\in\mathfrak{R}$ with $r_1=1$ and $C>0$ such that $\left\|(B\circ x_{j_{\beta}} p_{\beta}(x)) (x,D)\right\|_{\mathcal{L}_b(H^s)}$ $\leq C/R_{\beta}$.\\
\indent $b)$ The sets $\{b\}$ and $\ds\left\{h^{|\beta|}\xi_j p_{\beta}(x)\big|\, \beta\in\NN^d\right\}$ are $\left(e^{\tilde{M}(m|x|)}e^{\tilde{M}(m|\xi|)},(M_p)\right)$-conju\-gate (resp. $\left(e^{\tilde{N}_{k_p}(|x|)}e^{\tilde{N}_{k_p}(|\xi|)},\{M_p\}\right)$-conjugate). In particular, for every $\tilde{h}>0$ there exists $C>0$ such that $\left\|(B\circ p_{\beta}(x) \partial_j) (x,D)\right\|_{\mathcal{L}_b(H^s)}\leq C\tilde{h}^{|\beta|}$. Moreover, there exist $(r_p)\in\mathfrak{R}$ with $r_1=1$ and $C>0$ such that $\left\|(B\circ p_{\beta}(x)\partial_j) (x,D)\right\|_{\mathcal{L}_b(H^s)}$ $\leq C/R_{\beta}.$
\end{lemma}

\begin{proof} We prove $a)$, the proof of $b)$ being completely analogous. In the $(M_p)$ case, let $h,h'>0$ be arbitrary but fixed. We have
\beqs
h^{|\beta|}\left|D^{\alpha}\left(x^{e_{j_{\beta}}}p_{\beta}(x)\right)\right|&\leq& h^{|\beta|}\left|x_{j_{\beta}}\right|\left|D^{\alpha}p_{\beta}(x)\right|+ |\alpha|h^{|\beta|}\left|D^{\alpha-e_{j_{\beta}}}p_{\beta}(x)\right|\\
&=&S_1(x)+S_2(x).
\eeqs
Take $h_2<1$ such that $3h_2H\leq m/2$ and take $h_1<1$ such that $2Hh_1/h_2\leq h'$, $Hh_1\leq m/2$ and $h_1\leq 1/h$. To estimate $S_1(x)$ we have
\beqs
S_1(x)&\leq& C_1h^{|\beta|}|x|\frac{(Hh_1)^{|\alpha|}h_1^{|\beta|}A_{\alpha}e^{\tilde{M}(h_1|x|)}\langle (x,\xi)\rangle^{|\alpha|}}{\tilde{M}_{|\alpha|+1}\langle (x,\xi)\rangle^{\rho|\alpha|}}\\
&\leq& C_1\frac{(Hh_1)^{|\alpha|}A_{\alpha}e^{\tilde{M}(h_1|x|)}e^{\tilde{M}(h_2\langle(x,\xi)\rangle)}} {h_2^{|\alpha|+1} \langle (x,\xi)\rangle^{\rho|\alpha|}}\\
&\leq& C_2\frac{h'^{|\alpha|}A_{\alpha}e^{\tilde{M}((h_1+3h_2)H|x|)}e^{\tilde{M}(3h_2|\xi|)}}{\langle (x,\xi)\rangle^{\rho|\alpha|}}\leq C_2\frac{h'^{|\alpha|}A_{\alpha}e^{\tilde{M}(m|\xi|)}e^{\tilde{M}(m|x|)}}{\langle (x,\xi)\rangle^{\rho|\alpha|}}.
\eeqs
Similar estimates can be obtained for $S_2(x)$ in the same way and the $\{M_p\}$ case can be treated similarly. The estimate $\left\|(B\circ x_j p_{\beta}(x)) (x,D)\right\|_{\mathcal{L}_b(H^s)}\leq C\tilde{h}^{|\beta|}$ follows from Proposition \ref{330}. The last part can be proved similarly as in the proof of Lemma \ref{340}, by using Lemma 3.4 of \cite{Komatsu3}.
\end{proof}

\section{The proof of Theorem \ref{mainthm} }\label{sec2}

The proof of Theorem \ref{mainthm} needs some preparation. First of all it is useful to characterize the space $\SSS^*(\RR^d)$ in terms of suitable scales of Sobolev norms.

Namely, let $$
\|\varphi\|_{s,h}=\sum_{\alpha\in\NN^d}\frac{h^{|\alpha|}}{M_{\alpha}}\left\|x^{\alpha}\varphi(x)\right\|_{H^s} \quad \textrm{and} \quad
\|\varphi\|_{\{s,h\}}=\sum_{\alpha\in\NN^d}\frac{h^{|\alpha|}}{M_{\alpha}}\left\|D^{\alpha}\varphi(x)\right\|_{H^s}.$$
Moreover, for $h>0$ and $(r_p)\in\mathfrak{R}$, set
\beqs
H^{s,h}_N[\varphi]=\sum_{|\alpha|\leq N}\frac{h^{|\alpha|}}{M_{\alpha}}\left\|x^{\alpha}\varphi(x)\right\|_{H^s}&,&\, H^{s,h,(r_p)}_{N}[\varphi]=\sum_{|\alpha|\leq N}\frac{h^{|\alpha|}R_{\alpha}}{M_{\alpha}} \left\|x^{\alpha}\varphi(x)\right\|_{H^s},\\
E^{s,h}_N[\varphi]=\sum_{|\alpha|\leq N}\frac{h^{|\alpha|}}{M_{\alpha}}\left\|D^{\alpha}\varphi(x)\right\|_{H^s}&,&\,
E^{s,h,(r_p)}_{N}[\varphi]=\sum_{|\alpha|\leq N}\frac{h^{|\alpha|}R_{\alpha}}{M_{\alpha}}\left\|D^{\alpha}\varphi(x)\right\|_{H^s}.
\eeqs

Let 
$\varphi\in\SSS\left(\RR^d\right)$. We recall the  well known result (see for example \cite{CKK} or \cite{NR}).
\begin{lemma}\label{separation}
The following conditions are equivalent:
\begin{itemize}
\item[$i)$] $\varphi\in\SSS^*\left(\RR^d\right)$;
\item[$ii)$] there exists $s>d/2$ such that for every $h>0$ (resp. there exists $h>0$) such that $\|\varphi\|_{s,h}<\infty$ and $\|\varphi\|_{\{s,h\}}<\infty$.
\end{itemize}
\end{lemma}

By Lemma \ref{separation} we can prove that a function $u \in \SSS^*(\RR^d) $ by proving the decay and the regularity properties separately. This allows to simplify considerably the proofs, see also \cite{CGR3, CGR4}. \par
Next we state a preliminary technical result which will be used in the subsequent proofs.

\begin{lemma}\label{115}
Let $M_p$ be a sequence which satisfies $(M.3)'$, $(M.4)$ and $M_0=M_1=1$. Let $(k'_p),(k''_p)\in\mathfrak{R}$, $k'_1=k''_1=1$. There exists $(r'_p)\in\mathfrak{R}$ such that $r'_1=1$, $(r'_p)\leq (k'_p)$, $(r'_p)\leq (k''_p)$ and the sequence $N_p=M_p/\prod_{j=1}^p r'_p$, for $p\in\ZZ_+$ and $N_0=1$ satisfies  $(M.3)'$ and $(M.4)$.
\end{lemma}

\begin{proof} Let $a_p> 0$, $p\in\ZZ_+$, are such that $\ds \sum_{p=1}^{\infty}a_p$ is convergent. Then one easily verifies that $\ds \sum_{p=1}^{\infty}\frac{a_p}{s_p}$ is also convergent, where $\ds s_p=\sqrt{\sum_{j=p}^{\infty}a_j}$, $p\in\ZZ_+$ (one easily obtains that the partial sums of the series $\sum a_p/s_p$ are a Cauchy sequence). Put $\tilde{c}=\sqrt{\sum_{j=1}^{\infty} 1/m_j}$ and define $\ds \tilde{r}'_p=\tilde{c}\left(\sum_{j=p}^{\infty} \frac{1}{m_j}\right)^{-1/2}$, $p\in\ZZ_+$. Then we obtain that $\tilde{r}'_1=1$, $(\tilde{r}'_p)\in\mathfrak{R}$ and $\sum \tilde{r}'_p/m_p$ converges. Let $r_p=\min\{k'_p,k''_p,\tilde{r}'_p\}$, for $p\in\ZZ_+$. Then, obviously, $r_1=1$, $(r_p)\in\mathfrak{R}$, $(r_p)\leq (k'_p)$, $(r_p)\leq (k''_p)$ and $(r_p)\leq (\tilde{r}'_p)$. Also $\sum r_p/m_p$ converges. Define the sequence $(r'_p)$ by $r'_1=1$ and inductively
\beqs
r'_{p+1}=\min\left\{r_{p+1},\frac{pm_{p+1}}{(p+1)m_p} r'_p\right\},
\eeqs
for $p\in\ZZ_+$. We will prove that this $(r'_p)$ satisfies the desired conditions. First, note that $r'_p\leq r_p$, for all $p\in\ZZ_+$. Since $r_{p+1}\geq r_p$ and $pm_{p+1}\geq(p+1)m_p$ (which is equivalent to $(M.4)$ for $M_p$) it follows that
\beqs
r'_{p+1}=\min\left\{r_{p+1},\frac{pm_{p+1}}{(p+1)m_p} r'_p\right\}\geq \min\{r_p,r'_p\}=r'_p,
\eeqs
for all $p\in\ZZ_+$. To prove that $r'_p$ tends to infinity, assume the contrary. Since we already proved that $r'_p$ is monotonically increasing, there exists $C>1$ such that $r'_p\leq C$ for all $p\in\ZZ_+$. Since $(r_p)\in \mathfrak{R}$, there exists $p_0\in\ZZ_+$ such that $r_p>C+1$ for all $p\geq p_0$. But then, $\ds r'_{p+1}=\frac{pm_{p+1}}{(p+1)m_p} r'_p$ for all $p\geq p_0$. Then, for $p\geq p_0$, we have
\beqs
r'_{p+1}&=&\frac{pm_{p+1}}{(p+1)m_p} r'_p=\frac{pm_{p+1}}{(p+1)m_p}\cdot\frac{(p-1)m_p}{pm_{p-1}} r'_{p-1}\\
&=&...=\frac{pm_{p+1}}{(p+1)m_p}\cdot\frac{(p-1)m_p}{pm_{p-1}}\cdot...\cdot\frac{p_0 m_{p_0+1}}{(p_0+1)m_{p_0}} r'_{p_0}\\
&=&\frac{p_0 m_{p+1}}{(p+1)m_{p_0}}r'_{p_0}
\eeqs
which tends to infinity when $p\rightarrow\infty$ because of $(M.3)'$ for $M_p$. Hence $(r'_p)\in\mathfrak{R}$. The claim that $N_p$ satisfies $(M.4)$ is equivalent to $\ds r'_{p+1}\leq \frac{p m_{p+1}}{(p+1)m_p}r'_p$, which is true by the way we defined the sequence $(r'_p)$. Moreover, if we put $n_p=N_p/N_{p-1}$, then $n_p=m_p/r'_p\geq m_p/r_p$ and we know that $\sum r_p/m_p$ converges. Hence $N_p$ satisfies $(M.3)'$.
\end{proof}

After these preliminaries we can prove the following two results.

\begin{theorem}\label{3970}
Let $a\in\Gamma^{*,\infty}_{A_p,\rho}\left(\RR^{2d}\right)$ be $(\tilde{M}_p)$-elliptic (resp. $\{\tilde{M}_p\}$-elliptic) and let $f\in\SSS^*\left(\RR^d\right)$. Assume that $u\in H^s\left(\RR^d\right)$, $s>d/2$, is a solution of $Au=f+F[u]$, where $F[u]$ is defined by \eqref{343} and \eqref{nonlinearterm} (resp. \eqref{345} and \eqref{nonlinearterm}). Then we have $\|u\|_{s,h}<\infty$ for every $h>0$ (resp. for some $h>0$).
\end{theorem}

\begin{theorem}\label{3975}
Let $a\in\Gamma^{*,\infty}_{A_p,\rho}\left(\RR^{2d}\right)$ be $(\tilde{M}_p)$-elliptic (resp. $\{\tilde{M}_p\}$-elliptic) and let $f\in\SSS^*\left(\RR^d\right)$. Assume that $u\in H^s\left(\RR^d\right)$, $s>d/2$, is a solution of $Au=f+F[u]$, where $F[u]$ is defined by \eqref{343} and \eqref{nonlinearterm} (resp. \eqref{345} and \eqref{nonlinearterm}).\\
\indent i) If $F[u]$ is a finite sum, then we have $\|u\|_{\{s,h\}}<\infty$ for every $h>0$ (resp. for some $h>0$).\\
\indent ii) If $F[u]$ is infinite sum, then $\ds\sum_{\alpha} \frac{h^{|\alpha|}}{\tilde{M}_{\alpha}}\|\partial^{\alpha}u\|_{H^s}<\infty$ for some $h>0$ in the $(M_p)$ case (resp. $\ds\sum_{\alpha} \frac{h^{|\alpha|}\|\partial^{\alpha}u\|_{H^s}}{\tilde{M}_{\alpha}\prod_{j=1}^{|\alpha|}k_j}<\infty$ for some $h>0$ in the $\{M_p\}$ case).
\end{theorem}

Notice that by Lemma \ref{separation}, Theorem \ref{mainthm} follows directly from the combination of Theorems \ref{3970} and \ref{3975}. Let us prove the two latter results.\\

\begin{lemma}\label{367}
Let $A=a(x,D)$ be $(\tilde{M}_p)$-elliptic (resp. $\{\tilde{M}_p\}$-elliptic) operator and let $B$ be its parametrix. Then the following properties hold:\\
\indent $i)$ In the $(M_p)$ case, let $(r_p)\in\mathfrak{R}$ be the sequence in Lemma \ref{340}. Let $(r'_p)\in\mathfrak{R}$ be a sequence such that $(r'_p)\leq (r_p)$, $r'_1=1$ and the sequence $M_p/R'_p$ satisfies $(M.3)'$ and $(M.4)$. Then for each $0<\varepsilon<1$ there exists $h_0=h_0(\varepsilon)$ such that for every $0<h<h_0$
\beqs
\sum_{|\alpha|=1}^N\frac{h^{|\alpha|}R'_{\alpha}}{M_{\alpha}}\left\|B[A,x^{\alpha}]u\right\|_{H^s}\leq \varepsilon H^{s,h,(r'_p)}_{N-1}[u].
\eeqs
\indent $ii)$ In the $\{M_p\}$ case, for each $0<\varepsilon<1$ there exists $h_0=h_0(\varepsilon)$ such that for all $0<h<h_0(\varepsilon)$
\beqs
\sum_{|\alpha|=1}^N\frac{h^{|\alpha|}}{M_{\alpha}}\left\|B[A,x^{\alpha}]u\right\|_{H^s}\leq \varepsilon H^{s,h}_{N-1}[u].
\eeqs
\end{lemma}

\begin{proof} First we prove the $(M_p)$ case. (The existence of such sequence $(r'_p)\in\mathfrak{R}$ is given by Lemma \ref{115}.) For shorter notation, put $N_p=M_p/R'_p$, for $p\in\ZZ_+$ and $N_0=1$. Observe that
$$
x^{\alpha}Au(x)
=\sum_{\beta\leq\alpha}{\alpha\choose\beta}(-1)^{|\beta|}(D^{\beta}_{\xi}a)(x,D)(x^{\alpha-\beta}u(x)).
$$
So, we obtain
\beqs
B[A,x^{\alpha}]u=\sum_{\substack{\beta\leq \alpha\\ \beta\neq 0}} {\alpha\choose\beta}(-1)^{|\beta|+1}B(D^{\beta}_{\xi} a)(x,D) (x^{\alpha-\beta}u(x)).
\eeqs
By Lemma \ref{340}, there exists $C>0$ such that $\left\|B(D^{\beta}_{\xi} a)(x,D)\right\|_{\mathcal{L}_b(H^s)}\leq C N_{\beta}$. Let $0<\varepsilon<1$ be fixed. Choose $0<h_0<1/2$ such that $h_0<\varepsilon\left(2C\sum_{|\beta|=1}^{\infty}2^{-|\beta|+1}\right)^{-1}$. For $0<h<h_0$ we obtain\\
\\
$\ds \sum_{|\alpha|=1}^N\frac{h^{|\alpha|}}{N_{\alpha}}\left\|B[A,x^{\alpha}]u\right\|_{H^s}$
\beqs
&\leq&\sum_{|\alpha|=1}^N \frac{h^{|\alpha|}}{N_{\alpha}} \sum_{\substack{\beta\leq\alpha\\ \beta\neq 0}} {\alpha\choose\beta}\left\|B(D^{\beta}_{\xi} a) (x,D) x^{\alpha-\beta} u\right\|_{H^s}\\
&\leq& C\sum_{|\alpha|=1}^N \sum_{\substack{\beta\leq\alpha\\ \beta\neq 0}} \frac{h^{|\alpha|}N_{\beta}}{N_{\alpha}} {\alpha\choose\beta}\left\|x^{\alpha-\beta} u\right\|_{H^s}\leq C\sum_{|\beta|=1}^N h^{|\beta|} \sum_{\substack{\alpha\geq\beta\\ |\alpha|\leq N}} \frac{h^{|\alpha|-|\beta|}}{N_{\alpha-\beta}} \left\|x^{\alpha-\beta} u\right\|_{H^s}\\
&\leq&\varepsilon H^{s,h,(r'_p)}_{N-1}[u],
\eeqs
where in the third inequality, we used $(M.4)$ for $N_p$ and the fact $\ds {\alpha\choose\beta}\leq {|\alpha|\choose |\beta|}$. This completes the proof in the $(M_p)$ case. For the $\{M_p\}$ case, let $\varepsilon>0$. By Lemma \ref{340}, there exist $h_1,C>0$ such that $\left\|B(D^{\beta}_{\xi} a)(x,D)\right\|_{\mathcal{L}_b(H^s)}\leq C h_1^{|\beta|}M_{\beta}$. Choose $h_0>0$ such that $h_0h_1<1/2$ and $h_0h_1\leq \varepsilon\left(2C\sum_{|\beta|=1}^{\infty}2^{-|\beta|+1}\right)^{-1}$. Then, for $0<h<h_0$, similarly as before, we obtain
\beqs
\sum_{|\alpha|=1}^N\frac{h^{|\alpha|}}{M_{\alpha}}\left\|B[A,x^{\alpha}]u\right\|_{H^s}&\leq&C\sum_{|\beta|=1}^N (hh_1)^{|\beta|}\sum_{\substack{\alpha\geq\beta\\ |\alpha|\leq N}} \frac{h^{|\alpha|-|\beta|}}{M_{\alpha-\beta}} \left\|x^{\alpha-\beta} u\right\|_{H^s}\\
&\leq& \varepsilon H^{s,h}_{N-1}[u],
\eeqs
which completes the proof.
\end{proof}

\begin{lemma}\label{360}
Let $A=a(x,D)$ be $(\tilde{M}_p)$-elliptic (resp. $\{\tilde{M}_p\}$-elliptic) operator and let $B$ be its parametrix. Let $F[u]$ be defined by \eqref{343}, \eqref{nonlinearterm}  in the $(M_p)$ case (resp. by \eqref{345}, \eqref{nonlinearterm} in the $\{M_p\}$ case). Then the following properties hold:
\\
\indent i) In the $(M_p)$ case, let $(r_p)\in\mathfrak{R}$ be the sequence in Lemma \ref{350}. Let $(r'_p)\in\mathfrak{R}$ be a sequence such that $(r'_p)\leq (r_p)$, $r'_1=1$ and the sequence $M_p/R'_p$ satisfies $(M.3)'$ and $(M.4)$. Then for each $0<\varepsilon<1$ there exists $h_0=h_0(\varepsilon)$ such that for every $0<h<h_0$
\beqs
\sum_{|\alpha|=1}^N \frac{h^{|\alpha|}R'_{\alpha}}{M_{\alpha}}\left\|Bx^{\alpha}F[u]\right\|_{H^s}\leq \varepsilon H^{s,h,(r'_p)}_{N-1}[u].
\eeqs
\indent ii) In the $\{M_p\}$ case, for each $0<\varepsilon<1$ there exists $h_0=h_0(\varepsilon)$ such that for every $0<h<h_0$
\beqs
\sum_{|\alpha|=1}^N \frac{h^{|\alpha|}}{M_{\alpha}}\left\|Bx^{\alpha}F[u]\right\|_{H^s}\leq \varepsilon H^{s,h}_{N-1}[u].
\eeqs
\end{lemma}

\begin{proof} $i)$ Observe that the existence of such sequence $(r'_p)\in\mathfrak{R}$ is given by Lemma \ref{115}. Let $\alpha \in \NN^d$ with $|\alpha|\geq 1$ and let $j=j_{\alpha}\in\{1,...,d\}$ such that $\alpha_{j}>0$. By Lemma \ref{350}, there exists $C_1>0$ such that
\beqs
\left\|(B\circ x_{j}p_{\beta}(x))(x,D)\right\|_{\mathcal{L}_b(H^s)}\leq C_1/R'_{\beta}.
\eeqs
We obtain that $\ds\left\|B\left( x_{j}p_{\beta}(x)x^{\alpha-e_{j}}u^{|\beta|}\right)\right\|_{H^s}\leq \frac{C_1}{R'_{\beta}}\left\|x^{\alpha-e_{j}}u^{|\beta|}\right\|_{H^s}$. Moreover
\beqs
\left\|x^{\alpha-e_{j}}u^{|\beta|}\right\|_{H^s}\leq C_s^{|\beta|-1}\left\|x^{\alpha-e_{j}}u\right\|_{H^s}\|u\|_{H^s}^{|\beta|-1}.
\eeqs
Hence
\beqs
\left\|B\left( p_{\beta}(x)x^{\alpha}u^{|\beta|}\right)\right\|_{H^s}\leq C_1\frac{\left(C_s\|u\|_{H^s}\right)^{|\beta|-1}}{R'_{\beta}}\left\|x^{\alpha-e_{j}}u\right\|_{H^s}.
\eeqs
Let $\ds C_2=\sum_{|\beta|=2}^{\infty}\frac{\left(C_s\|u\|_{H^s}\right)^{|\beta|-1}}{R'_{\beta}}$. We obtain
\beqs
\sum_{|\alpha|=1}^N\frac{h^{|\alpha|}}{N_{\alpha}}\left\|Bx^{\alpha}F[u]\right\|_{H^s}&\leq& \sum_{|\alpha|=1}^N\sum_{|\beta|=2}^{\infty} \frac{h^{|\alpha|}}{N_{\alpha}}\left\|B(x^{\alpha}p_{\beta}(x)u^{|\beta|})\right\|_{H^s}\\
&\leq& C_1C_2h\sum_{|\alpha|=1}^N\frac{h^{|\alpha|-1}}{N_{\alpha-e_{j}}} \left\|x^{\alpha-e_{j_{\alpha}}}u\right\|_{H^s}\leq C_3 hH^{s,h,(r'_p)}_{N-1}[u].
\eeqs
Moreover, for fixed $0<\varepsilon<1$, since $C_3$ does not depend on $h$, we can find $h_0=h_0(\varepsilon)$ such that for all $0<h<h_0$
\beqs
\sum_{|\alpha|=1}^N\frac{h^{|\alpha|}}{N_{\alpha}}\left\|Bx^{\alpha}F[u]\right\|_{H^s}\leq \varepsilon H^{s,h,(r'_p)}_{N-1}[u],
\eeqs
which complete the proof in the $(M_p)$ case. \\
\indent $ii)$ In the $\{M_p\}$ case by using Lemma \ref{350}, one similarly obtains that for every $\tilde{h}>0$ there exists $C_1>0$ such that
\beqs
\left\|B\left( p_{\beta}(x)x^{\alpha}u^{|\beta|}\right)\right\|_{H^s}\leq C_1\left(\tilde{h}C_s\|u\|_{H^s}\right)^{|\beta|-1}\left\|x^{\alpha-e_{j_{\alpha}}}u\right\|_{H^s}.
\eeqs
Fix $\tilde{h}$ such that $\tilde{h}C_s\|u\|_{H^s}<1/2$. We have
\beqs
\sum_{|\alpha|=1}^N\frac{h^{|\alpha|}}{M_{\alpha}}\left\|Bx^{\alpha}F[u]\right\|_{H^s}\leq C'_3 hH^{s,h}_{N-1}[u],
\eeqs
for a constant $C'_3$ which is the same for all $h$. Hence, we obtain the claim in the $\{M_p\}$ case.
\end{proof}

\noindent
\textit{Proof of Theorem \ref{3970}.} Fixed $\alpha \in \NN^d$  let us multiply both members of \eqref{maineq} by $x^{\alpha}$. We have $x^{\alpha}Au=x^{\alpha}f+x^{\alpha}F[u]$. Then, introducing commutators we get
\begin{equation} \label{pippo}
A(x^{\alpha}u)=[A,x^{\alpha}]u+x^{\alpha}f+x^{\alpha}F[u].
\end{equation}
By applying the parametrix $B$ of $A$ to both sides of \eqref{pippo} we have
\beq\label{363}
x^{\alpha}u=B[A,x^{\alpha}]u+B(x^{\alpha}f)+B(x^{\alpha}F[u])+T(x^{\alpha}u)
\eeq
for some *-regularizing operator $T$.
We first consider the $(M_p)$ case. Since $f\in\SSS^{(M_p)}$, for every $\tilde{h}>0$ we have $\ds\sup_{\alpha}\frac{\|x^{\alpha} f\|_{H^s}}{\tilde{h}^{|\alpha|}M_{\alpha}}<\infty$. Hence, by Lemma 3.4 of \cite{Komatsu3}, there exist $(\tilde{r}_p)\in\mathfrak{R}$ and $C'>0$ such that $\|x^{\alpha} f\|_{H^s}\leq C'M_{\alpha}/\tilde{R}_{\alpha}$. Obviously, without loss of generality, we can assume that $\tilde{r}_1=1$. By Lemma \ref{115} we can find $(r'_p)\in\mathfrak{R}$ such that $r'_1=1$, $(r'_p)\leq (\tilde{r}_p)$, $(r'_p)$ is smaller than the sequences in Lemmas \ref{340} and \ref{350}, and the sequence $N_p=M_p/R'_p$, for $p\in\ZZ_+$ and $N_0=1$, satisfies $(M.3)'$, $(M.4)$ and $N_1=1$. If we multiply (\ref{363}) by $h^{|\alpha|}/N_{\alpha}$, take Sobolev norms and sum up for $|\alpha|\leq N$, we obtain
\beqs
H^{s,h,(r'_p)}_N[u]&\leq&\|u\|_{H^s}+\sum_{|\alpha|=1}^N \frac{h^{|\alpha|}}{N_{\alpha}}\|B[A,x^{\alpha}]u\|_{H^s}+\sum_{|\alpha|=1}^N \frac{h^{|\alpha|}}{N_{\alpha}}\|B(x^{\alpha}f)\|_{H^s}\\
&{}&+\sum_{|\alpha|=1}^N \frac{h^{|\alpha|}}{N_{\alpha}}\|B(x^{\alpha}F[u])\|_{H^s}+\sum_{|\alpha|=1}^N \frac{h^{|\alpha|}}{N_{\alpha}}\|T(x^{\alpha}u)\|_{H^s}.
\eeqs
We will estimate each of the terms above. First, since $B$ is bounded on $H^s$, there exists $C''>0$ such that $\left\|B(x^{\alpha} f)\right\|_{H^s}\leq C''\|x^{\alpha}f\|_{H^s}$, from what we obtain $\ds \sum_{|\alpha|=1}^N \frac{h^{|\alpha|}}{N_{\alpha}}\|B(x^{\alpha}f)\|_{H^s}\leq C'''\sum_{|\alpha|=1}^{\infty}\frac{1}{2^{|\alpha|}}=C_1$ for all $0<h<1/2$. To estimate the sum with $T(x^{\alpha}u)$, since $|\alpha|>0$ there exists $j=j_{\alpha}\in\{1,...,d\}$ such that $\alpha_{j}\geq 1$. Hence, there exists $C_2>0$ such that $\left\|T\circ x_j\right\|_{\mathcal{L}_b(H^s)}\leq C_2$. Then we obtain
\beqs
\sum_{|\alpha|=1}^N \frac{h^{|\alpha|}}{N_{\alpha}}\|T(x^{\alpha}u)\|_{H^s}\leq C_2h\sum_{|\alpha|=1}^N \frac{h^{|\alpha|-1}}{N_{|\alpha|-1}}\left\|x^{\alpha-e_{j_{\alpha}}}u\right\|_{H^s}\leq C_3h H^{s,h,(r'_p)}_{N-1}[u].
\eeqs
Since $C_3$ does not depend on $h$, for fixed $0<\varepsilon<1$ we can find $h_0=h_0(\varepsilon)<1/2$ such that for all $0<h<h_0$
\beqs
\sum_{|\alpha|=1}^N \frac{h^{|\alpha|}}{N_{\alpha}}\|T(x^{\alpha}u)\|_{H^s}\leq \varepsilon H^{s,h,(r'_p)}_{N-1}[u].
\eeqs
Now, if we use Lemmas \ref{360} and \ref{367} for fixed $0<\varepsilon <1$ there exists $h_0=h_0(\varepsilon)$ such that for all $0<h<h_0$ we obtain
\beqs
H^{s,h,(r'_p)}_{N}[u]\leq \|u\|_{H^s}+\varepsilon H^{s,h,(r'_p)}_{N-1}[u]+C_1+\varepsilon H^{s,h,(r'_p)}_{N-1}[u]+\varepsilon H^{s,h,(r'_p)}_{N-1}[u].
\eeqs
By iterating this estimate and possibly shrinking $\varepsilon$ we obtain that
$\ds\sum_{|\alpha|=0}^{\infty}\frac{h^{|\alpha|}}{N_{\alpha}}\|x^{\alpha}u\|_{H^s}$ is finite for some sufficiently small $h$. If $\tilde{h}>0$ is arbitrary but fixed, there exists $\tilde{C}>0$ such that $\tilde{h}^p\leq \tilde{C} h^p R'_p$, for all $p\in\ZZ_+$. Hence the sum $\ds \sum_{|\alpha|=0}^{\infty}\frac{\tilde{h}^{|\alpha|}}{M_{\alpha}}\|x^{\alpha}u\|_{H^s}$ converges. This completes the proof in the $(M_p)$ case. The $\{M_p\}$ case is completely similar. We leave the details to the reader.
\qed

\vskip0.3cm
Now we prove Theorem \ref{3975}.

\begin{lemma}\label{3910}
Let $A=a(x,D)$ be $(\tilde{M}_p)$-elliptic (resp. $\{\tilde{M}_p\}$-elliptic) operator and let $B$ be its parametrix. Then the following properties hold: \\
\indent $i)$ In the $(M_p)$ case, let $(r_p)\in\mathfrak{R}$ be the sequence in Lemma \ref{340}. Let $(r'_p)\in\mathfrak{R}$ be a sequence such that $(r'_p)\leq (r_p)$, $r'_1=1$ and the sequence $M_p/R'_p$ satisfies $(M.3)'$ and $(M.4)$. Then for each $0<\varepsilon<1$ there exists $h_0=h_0(\varepsilon)$ such that for every $0<h<h_0$
\beqs
\sum_{|\alpha|=1}^N\frac{h^{|\alpha|}R'_{\alpha}}{M_{\alpha}}\left\|B[A,\partial^{\alpha}]u\right\|_{H^s}\leq \varepsilon E^{s,h,(r'_p)}_{N-1}[u].
\eeqs
\indent $ii)$ In the $\{M_p\}$ case, for each $0<\varepsilon<1$ there exists $h_0=h_0(\varepsilon)$ such that for all $0<h<h_0(\varepsilon)$
\beqs
\sum_{|\alpha|=1}^N\frac{h^{|\alpha|}}{M_{\alpha}}\left\|B[A,\partial^{\alpha}]u\right\|_{H^s}\leq \varepsilon E^{s,h}_{N-1}[u].
\eeqs
\end{lemma}

\begin{proof} First we prove the $(M_p)$ case. As before, put $N_p=M_p/R'_p$, for $p\in\ZZ_+$ and $N_0=1$ (the existence of such $(r'_p)\in\mathfrak{R}$ is given by Lemma \ref{115}). Observe that
\beqs
B[A,\partial^{\alpha}]u=-\sum_{\substack{\beta\leq\alpha\\ \beta\neq 0}}{\alpha\choose\beta}B(\partial^{\beta}_x a)(x,D) \partial^{\alpha-\beta}_x u.
\eeqs
By Lemma \ref{340}, there exists $C>0$ such that $\left\|B(\partial^{\beta}_x a)(x,D)\right\|_{\mathcal{L}_b(H^s)}\leq C N_{\beta}$. Let $0<\varepsilon<1$ be fixed. Choose $0<h_0<1/2$ such that $h_0<\varepsilon \left(2C\sum_{|\beta|=1}^{\infty}2^{-|\beta|+1}\right)^{-1}$. For $0<h<h_0$ we obtain\\
\\
$\ds \sum_{|\alpha|=1}^N\frac{h^{|\alpha|}}{N_{\alpha}}\left\|B[A,\partial^{\alpha}]u\right\|_{H^s}$
\beqs
&\leq&\sum_{|\alpha|=1}^N \frac{h^{|\alpha|}}{N_{\alpha}} \sum_{\substack{\beta\leq\alpha\\ \beta\neq 0}} {\alpha\choose\beta}\left\|B(\partial^{\beta}_x a) (x,D) \partial^{\alpha-\beta}_x u\right\|_{H^s}\\
&\leq& C\sum_{|\alpha|=1}^N \sum_{\substack{\beta\leq\alpha\\ \beta\neq 0}} \frac{h^{|\alpha|}N_{\beta}}{N_{\alpha}} {\alpha\choose\beta}\left\|\partial^{\alpha-\beta}_x u\right\|_{H^s}\leq C\sum_{|\beta|=1}^N h^{|\beta|}\sum_{\substack{\alpha\geq\beta\\ |\alpha|\leq N}} \frac{h^{|\alpha|-|\beta|}}{N_{\alpha-\beta}} \left\|\partial^{\alpha-\beta}_x u\right\|_{H^s}\\
&\leq&\varepsilon E^{s,h,(r'_p)}_{N-1}[u],
\eeqs
where in the third inequality, we used $(M.4)$ for $N_p$ and the fact $\ds {\alpha\choose\beta}\leq {|\alpha|\choose |\beta|}$. This completes the proof in the $(M_p)$ case. For the $\{M_p\}$ case, let $\varepsilon>0$. By Lemma \ref{340}, there exist $h_1,C>0$ such that $\left\|B(\partial^{\beta}_x a)(x,D)\right\|_{\mathcal{L}_b(H^s)}\leq C h_1^{|\beta|}M_{\beta}$. Choose $h_0>0$ such that $h_0h_1<1/2$ and $h_0h_1\leq \varepsilon \left(2C\sum_{|\beta|=1}^{\infty}2^{-|\beta|+1}\right)^{-1}$. Then, for $0<h<h_0$, similarly as before, we obtain
\beqs
\sum_{|\alpha|=1}^N\frac{h^{|\alpha|}}{M_{\alpha}}\left\|B[A,\partial^{\alpha}]u\right\|_{H^s}&\leq&C\sum_{|\beta|=1}^N (hh_1)^{|\beta|}\sum_{\substack{\alpha\geq\beta\\ |\alpha|\leq N}} \frac{h^{|\alpha|-|\beta|}}{M_{\alpha-\beta}} \left\|\partial^{\alpha-\beta}_x u\right\|_{H^s}\\
&\leq& \varepsilon E^{s,h}_{N-1}[u],
\eeqs
which completes the proof.
\end{proof}

\begin{lemma}\label{3920}
Let $A=a(x,D)$ be $(\tilde{M}_p)$-elliptic (resp. $\{\tilde{M}_p\}$-elliptic) operator and let $B$ be its parametrix. Let $F[u]=p(x)u^l$, for some $l\geq 2$, $l\in\NN$. Then the following properties hold: \\
\indent $i)$ In the $(M_p)$ case, let $(r_p)\in\mathfrak{R}$ be the sequence in Lemma \ref{350}. Let $(r'_p)\in\mathfrak{R}$ be a sequence such that $(r'_p)\leq (r_p)$, $r'_1=1$ and the sequence $M_p/R'_p$ satisfies $(M.3)'$ and $(M.4)$. Then for each $0<\varepsilon<1$ there exists $h_0=h_0(\varepsilon)$ such that for every $0<h<h_0$
\beqs
\sum_{|\alpha|=1}^N \frac{h^{|\alpha|}R'_{\alpha}}{M_{\alpha}}\left\|B\left(\partial^{\alpha}(p(x)u^l)\right)\right\|_{H^s}\leq \varepsilon\left(E^{s,h,(r'_p)}_{N-1}[u]\right)^l.
\eeqs
\indent $ii)$ In the $\{M_p\}$ case, for each $0<\varepsilon<1$ there exists $h_0=h_0(\varepsilon)$ such that for every $0<h<h_0$
\beqs
\sum_{|\alpha|=1}^N \frac{h^{|\alpha|}}{M_{\alpha}}\left\|B\left(\partial^{\alpha}(p(x)u^l)\right)\right\|_{H^s}\leq \varepsilon\left(E^{s,h}_{N-1}[u]\right)^l.
\eeqs
\end{lemma}

\begin{proof} Observe that
\beqs
B\left(\partial^{\alpha}(p(x)u^l)\right)=B\left(p(x)\partial^{\alpha}u^l\right) +\sum_{\substack{\gamma\leq \alpha\\ \gamma\neq 0}} {\alpha\choose\gamma} B\left(\partial^{\gamma}p(x) \partial^{\alpha-\gamma}u^l\right)
\eeqs
First we consider the $(M_p)$ case. As before, put $N_p=M_p/R'_p$, for $p\in\ZZ_+$ and $N_0=1$. Since $|\alpha|\geq 1$, there exists $j=j_{\alpha}\in\{1,...,d\}$ such that $\alpha_{j}>0$. By Lemma \ref{350}, there exists $C_1>0$ such that $\left\|\left(B\circ p(x)\partial_j\right) (x,D)\right\|_{\mathcal{L}_b(H^s)}\leq C_1$. Then we have
\beqs
\left\|B\left(p(x)\partial^{\alpha}u^l\right)\right\|_{H^s}&\leq& C_1\left\|\partial^{\alpha-e_{j}}u^l\right\|_{H^s}\\
&\leq& C_2\sum_{\nu^{(1)}+...+\nu^{(l)}=\alpha-e_{j}}\frac{(\alpha-e_{j})!}{\nu^{(1)}!\cdot...\cdot\nu^{(l)}!} \prod_{k=1}^l\left\|\partial^{\nu^{(k)}}u\right\|_{H^s}.
\eeqs
Observe that, by $(M.4)$,
\beqs
\frac{(\alpha-e_{j})!}{\nu^{(1)}!\cdot...\cdot\nu^{(l)}!}\cdot\frac{h^{|\alpha|}}{N_{\alpha}}\leq \frac{h N_{|\alpha|-1}} {N_{\alpha}} \prod_{k=1}^l\frac{h^{\left|\nu^{(k)}\right|}}{N_{\nu^{(k)}}}\leq h \prod_{k=1}^l\frac{h^{\left|\nu^{(k)}\right|}}{N_{\nu^{(k)}}}.
\eeqs
We obtain
\beqs
\sum_{|\alpha|=1}^N\frac{h^{|\alpha|}}{N_{\alpha}}\left\|B\left(p(x)\partial^{\alpha}u^l\right)\right\|_{H^s}&\leq& C_2h\sum_{|\alpha|=1}^N \sum_{\nu^{(1)}+...+\nu^{(l)}=\alpha-e_{j_{\alpha}}}\prod_{k=1}^l \frac{h^{\left|\nu^{(k)}\right|}}{N_{\nu^{(k)}}}\left\|\partial^{\nu^{(k)}} u\right\|_{H^s}\\
&\leq& dC_2h\left(E^{s,h,(r'_p)}_{N-1}[u]\right)^l.
\eeqs
Since $C_2$ does not depend on $h$, for fixed $0<\varepsilon<1$ we can take $h_0=\varepsilon/(dC_2)$. Then for all $h<h_0$ we obtain $\ds \sum_{|\alpha|=1}^N\frac{h^{|\alpha|}}{N_{\alpha}}\left\|B\left(p(x)\partial^{\alpha}u^l\right)\right\|_{H^s}\leq \varepsilon \left(E^{s,h,(r'_p)}_{N-1}[u]\right)^l$. One easily verifies that the functions $p'_{\beta}(x)=\partial^{\beta}_x p(x)$, $\beta\in\NN^d$ satisfy (\ref{343}). Lemma \ref{350} implies that there exists $C_1>0$ such that $\left\|B\circ\partial^{\gamma}p(x)\right\|_{\mathcal{L}_b(H^s)}\leq C_1/R'_{\gamma}$. For $h<1$ we obtain\\
\\
$\ds\sum_{|\alpha|=1}^N \frac{h^{|\alpha|}}{N_{\alpha}}\sum_{\substack{\gamma\leq \alpha\\ \gamma\neq 0}} {\alpha\choose\gamma} \left\|B\left(\partial^{\gamma}p(x) \partial^{\alpha-\gamma}u^l\right)\right\|_{H^s}$
\beqs
&\leq& C_1h\sum_{|\alpha|=1}^N \sum_{\substack{\gamma\leq \alpha\\ \gamma\neq 0}}{\alpha\choose\gamma} \frac{h^{|\alpha|-1}}{N_{\alpha}R'_{\gamma}}\left\|\partial^{\alpha-\gamma}u^l\right\|_{H^s}\\
&\leq& C'_2h\sum_{|\alpha|=1}^N \sum_{\substack{\gamma\leq \alpha\\ \gamma\neq 0}}{\alpha\choose\gamma} \frac{h^{|\alpha|-1}}{N_{\alpha}R'_{\gamma}}\sum_{\nu^{(1)}+...+\nu^{(l)}=\alpha-\gamma}\frac{(\alpha-\gamma)!} {\nu^{(1)}!\cdot...\cdot\nu^{(l)}!} \prod_{k=1}^l\left\|\partial^{\nu^{(k)}}u\right\|_{H^s}\\
&\leq&C'_2h\sum_{|\gamma|=1}^N \frac{1}{R'_{\gamma}}\sum_{\substack{\alpha\geq \gamma\\ |\alpha|\leq N}} \frac{h^{|\alpha|-1}}{N_{\alpha}}\sum_{\nu^{(1)}+...+\nu^{(l)}=\alpha-\gamma} \frac{\alpha!} {\nu^{(1)}!\cdot...\cdot\nu^{(l)}!\gamma!} \prod_{k=1}^l\left\|\partial^{\nu^{(k)}}u\right\|_{H^s}\\
&\leq&C'_2h\sum_{|\gamma|=1}^N \frac{1}{N_{\gamma}R'_{\gamma}}\sum_{\substack{\alpha\geq \gamma\\ |\alpha|\leq N}} \sum_{\nu^{(1)}+...+\nu^{(l)}=\alpha-\gamma} \prod_{k=1}^l\frac{h^{\left|\nu^{(k)}\right|}}{N_{\nu^{(k)}}}\left\|\partial^{\nu^{(k)}}u\right\|_{H^s}\\
&\leq& C'_2h\sum_{|\gamma|=1}^N \frac{1}{M_{\gamma}} \sum_{\left|\nu^{(1)}\right|+...+\left|\nu^{(l)}\right|=0}^{N-1} \prod_{k=1}^l\frac{h^{\left|\nu^{(k)}\right|}}{N_{\nu^{(k)}}}\left\|\partial^{\nu^{(k)}}u\right\|_{H^s}\leq C'_3h\left(E^{s,h,(r'_p)}_{N-1}[u]\right)^l.
\eeqs
Since $C'_3$ does not depend on $h$, for fixed $0<\varepsilon<1$ we can choose $h_0=\varepsilon/C'_3$. Then, for all $0<h<h_0$, we have
\beqs
\sum_{|\alpha|=1}^N \frac{h^{|\alpha|}}{N_{\alpha}}\sum_{\substack{\gamma\leq \alpha\\ |\gamma|\neq 0}} {\alpha\choose\gamma} \left\|B\left(\partial^{\gamma}p(x) \partial^{\alpha-\gamma}u^l\right)\right\|_{H^s}\leq \varepsilon\left(E^{s,h,(r'_p)}_{N-1}[u]\right)^l,
\eeqs
which, combined with the above estimate for $\ds \sum_{|\alpha|=1}^N\frac{h^{|\alpha|}}{N_{\alpha}}\left\|B\left(p(x)\partial^{\alpha}u^l\right)\right\|_{H^s}$, completes the proof in the $(M_p)$ case by shrinking $h_0$ if necessary. In the $\{M_p\}$ case the proof is similar.
\end{proof}

\textit{Proof of Theorem \ref{3975}.} $i)$ When $F[u]$ is a finite sum it is clear that it is enough to prove the theorem when $F[u]=p(x)u^l$, $l\geq 2$, $l\in\NN$. Differentiating both terms of \eqref{maineq}, we have $\partial^{\alpha}Au=\partial^{\alpha}f+\partial^{\alpha}F[u]$, from which we obtain
\beqs
A(\partial^{\alpha}u)=[A,\partial^{\alpha}]u+\partial^{\alpha}f+\partial^{\alpha}F[u].
\eeqs
Hence, we have
\beq\label{390}
\partial^{\alpha}u=B[A,\partial^{\alpha}]u+B(\partial^{\alpha}f)+B(\partial^{\alpha}F[u])+T(\partial^{\alpha}u).
\eeq
We consider the $(M_p)$ case. Since $f\in\SSS^{(M_p)}$, for every $\tilde{h}>0$, $\ds\sup_{\alpha}\frac{\|\partial^{\alpha} f\|_{H^s}}{\tilde{h}^{|\alpha|}M_{\alpha}}$ is bounded. Hence, by Lemma 3.4 of \cite{Komatsu3}, there exist $(\tilde{r}_p)\in\mathfrak{R}$ and $C'>0$ such that $\|\partial^{\alpha} f\|_{H^s}\leq C'M_{\alpha}/\tilde{R}_{\alpha}$. Obviously, without losing generality, we can assume that $\tilde{r}_1=1$. By Lemma \ref{115} we can find $(r'_p)\in\mathfrak{R}$ such that $r'_1=1$, $(r'_p)\leq (\tilde{r}_p)$, $(r'_p)$ is smaller than the sequences in Lemmas \ref{340} and \ref{350} and the sequence $N_p=M_p/R'_p$, for $p\in\ZZ_+$ and $N_0=1$, satisfies $(M.3)'$, $(M.4)$ and $N_1=1$. If we multiply (\ref{390}) by $h^{|\alpha|}/N_{\alpha}$, take Sobolev norms and sum up for $|\alpha|\leq N$, we obtain
\beqs
E^{s,h,(r'_p)}_N[u]&\leq&\|u\|_{H^s}+\sum_{|\alpha|=1}^N \frac{h^{|\alpha|}}{N_{\alpha}}\|B[A,\partial^{\alpha}]u\|_{H^s}+\sum_{|\alpha|=1}^N \frac{h^{|\alpha|}}{N_{\alpha}}\|B(\partial^{\alpha}f)\|_{H^s}\\
&{}&+\sum_{|\alpha|=1}^N \frac{h^{|\alpha|}}{N_{\alpha}}\|B(\partial^{\alpha}F[u])\|_{H^s}+\sum_{|\alpha|=1}^N \frac{h^{|\alpha|}}{N_{\alpha}}\|T(\partial^{\alpha}u)\|_{H^s}.
\eeqs
We estimate each of the terms above. By the growth estimate for the symbol of $B$ (\ref{grwww}) there exists $C''>0$ such that $\left\|B(\partial^{\alpha} f)\right\|_{H^s}\leq C''\|\partial^{\alpha}f\|_{H^s}$. Hence $\ds \sum_{|\alpha|=1}^N \frac{h^{|\alpha|}}{N_{\alpha}}\|B(\partial^{\alpha}f)\|_{H^s}\leq C'''\sum_{|\alpha|=1}^{\infty}\frac{1}{2^{|\alpha|}}=C_1$ for all $0<h<1/2$. To estimate the sum with $T(\partial^{\alpha}u)$, since $|\alpha|>0$ there exists $j_{\alpha}\in\{1,...,d\}$ such that $\alpha_{j_{\alpha}}\geq 1$. Hence there exists $C_2>0$ such that $\left\|T\circ\partial^{e_{j_{\alpha}}}\right\|_{\mathcal{L}_b(H^s)}\leq C_2$. Then we obtain
\beqs
\sum_{|\alpha|=1}^N \frac{h^{|\alpha|}}{N_{\alpha}}\|T(\partial^{\alpha}u)\|_{H^s}\leq C_2h\sum_{|\alpha|=1}^N \frac{h^{|\alpha|-1}}{N_{|\alpha|-1}}\left\|\partial^{\alpha-e_{j_{\alpha}}}u\right\|_{H^s}\leq C_3h E^{s,h,(r'_p)}_{N-1}[u].
\eeqs
Since $C_3$ does not depend on $h$, for fixed $0<\varepsilon<1$ we can find $h_0=h_0(\varepsilon)<1/2$ such that for all $0<h<h_0$
\beqs
\sum_{|\alpha|=1}^N \frac{h^{|\alpha|}}{N_{\alpha}}\|T(\partial^{\alpha}u)\|_{H^s}\leq \varepsilon E^{s,h,(r'_p)}_{N-1}[u].
\eeqs
For fixed $0<\varepsilon<1$, by Lemmas \ref{3910} and \ref{3920} for the chosen $(r'_p)$, we can find $h_0=h_0(\varepsilon)<1/2$ such that for all $0<h<h_0$, we have
\beqs
E^{s,h,(r'_p)}_N[u]&\leq&\|u\|_{H^s}+\varepsilon E^{s,h,(r'_p)}_{N-1}[u]+C_1 +\varepsilon\left(E^{s,h,(r'_p)}_{N-1}[u]\right)^l+\varepsilon E^{s,h,(r'_p)}_{N-1}[u].
\eeqs
By iterating this estimate one obtains that $\|\partial^{\alpha} u\|_{H^s}$, $\alpha\in\NN^d$, are finite and by shrinking $\varepsilon$ if necessary, that the sum $\ds \sum_{|\alpha|=0}^{\infty} \frac{h^{|\alpha|}}{N_{\alpha}}\|\partial^{\alpha} u\|_{H^s}$ converges, for some, small enough, $h$. If $\tilde{h}>0$ is arbitrary but fixed there exists $\tilde{C}>0$ such that $\tilde{h}^p\leq C h^p R'_p$ for all $p\in\ZZ_+$. Hence $\ds \sum_{|\alpha|=0}^{\infty} \frac{\tilde{h}^{|\alpha|}}{M_{\alpha}}\|\partial^{\alpha} u\|_{H^s}$ converges. This completes the proof in the $(M_p)$ case. The proof in the $\{M_p\}$ case is similar and we omit it.\\
\indent To prove $ii)$ we consider first the $(M_p)$ case. Proceed as in the proof $i)$ to obtain
\beqs
\sum_{|\alpha|\leq N} \frac{h^{|\alpha|}}{\tilde{M}_{\alpha}}\|\partial^{\alpha}u\|_{H^s} &\leq&\|u\|_{H^s}+\sum_{|\alpha|=1}^N \frac{h^{|\alpha|}}{\tilde{M}_{\alpha}}\|B[A,\partial^{\alpha}]u\|_{H^s}+\sum_{|\alpha|=1}^N \frac{h^{|\alpha|}}{\tilde{M}_{\alpha}}\|B(\partial^{\alpha}f)\|_{H^s}\\
&{}&+\sum_{|\alpha|=1}^N \frac{h^{|\alpha|}}{\tilde{M}_{\alpha}}\|B(\partial^{\alpha}F[u])\|_{H^s}+\sum_{|\alpha|=1}^N \frac{h^{|\alpha|}}{\tilde{M}_{\alpha}}\|T(\partial^{\alpha}u)\|_{H^s}.
\eeqs
By Lemma \ref{349}, there exists $l\geq 1$ such that for each $\tilde{h}>0$ there exists $C_1>0$ such that $\left\|(B\circ D^{\alpha}p_{\beta}(x) \partial^{\gamma}) (x,D)\right\|_{\mathcal{L}(H^s)}\leq C_1\tilde{h}^{|\beta|}l^{|\gamma|}\tilde{M}_{\gamma}$. For $0<h<1/(4l)$, we have
\beqs
\sum_{|\alpha|=1}^N \frac{h^{|\alpha|}}{\tilde{M}_{\alpha}}\left\|B(\partial^{\alpha}F[u])\right\|_{H^s}&\leq& \sum_{|\beta|=2}^{\infty}\sum_{|\alpha|=1}^N \sum_{\gamma\leq \alpha}{\alpha\choose\gamma} \frac{h^{|\alpha|}}{\tilde{M}_{\alpha}}\left\|B(\partial^{\gamma}p_{\beta} \partial^{\alpha-\gamma}u^{|\beta|})\right\|_{H^s}\\
&\leq& C_1 \sum_{|\beta|=2}^{\infty}\sum_{|\alpha|=1}^N \frac{\tilde{h}^{|\beta|}\left\|u^{|\beta|}\right\|_{H^s}}{2^{|\alpha|}}\leq C_2,
\eeqs
where in the last inequality we used that $\left\|u^{|\beta|}\right\|_{H^s}\leq C_s^{|\beta|-1}\|u\|^{|\beta|}_{H^s}$ and chose $\tilde{h}\leq 1/(2C_s\|u\|_{H^s})$. The sequence $\tilde{M}_p$ satisfies $(M.4)$, so by analogous technique as in the proof of Lemma \ref{3910} one can prove that for each $0<\varepsilon<1$ there exists $h_0=h_0(\varepsilon)<1/2$ such that for every $0<h<h_0$
\beqs
\sum_{|\alpha|=1}^N\frac{h^{|\alpha|}}{\tilde{M}_{\alpha}}\left\|B[A,\partial^{\alpha}]u\right\|_{H^s}\leq \varepsilon \sum_{|\alpha|\leq N-1} \frac{h^{|\alpha|}}{\tilde{M}_{\alpha}}\left\|\partial^{\alpha} u\right\|_{H^s}.
\eeqs
Also, similarly as in the proof of $i)$, we have that for $0<h<h_0(\varepsilon)$, $$\ds \sum_{|\alpha|=1}^N \frac{h^{|\alpha|}}{\tilde{M}_{\alpha}}\|B(\partial^{\alpha}f)\|_{H^s}\leq C_3$$ and
\beqs
\sum_{|\alpha|=1}^N \frac{h^{|\alpha|}}{\tilde{M}_{\alpha}}\|T(\partial^{\alpha}u)\|_{H^s}\leq \varepsilon \sum_{|\alpha|\leq N-1} \frac{h^{|\alpha|}}{\tilde{M}_{\alpha}}\|\partial^{\alpha} u\|_{H^s}.
\eeqs
Hence, for $0<h<h_0(\varepsilon)$, for sufficiently small $h_0(\varepsilon)$, we have
\beqs
\sum_{|\alpha|\leq N} \frac{h^{|\alpha|}}{\tilde{M}_{\alpha}}\|\partial^{\alpha} u\|_{H^s}\leq\|u\|_{H^s}+ C_4+2\varepsilon\sum_{|\alpha|\leq N-1} \frac{h^{|\alpha|}}{\tilde{M}_{\alpha}}\|\partial^{\alpha} u\|_{H^s}.
\eeqs
By iterating this estimate and possibly shrinking $\varepsilon$ we obtain that
$\ds\sum_{|\alpha|=0}^{\infty}\frac{h^{|\alpha|}}{\tilde{M}_{\alpha}}\|\partial^{\alpha} u\|_{H^s}$ is finite for some sufficiently small $h$, which finishes the proof in the $(M_p)$ case. The $\{M_p\}$ case is completely analogous.
\qed


\begin{thebibliography}{999}
\setlength{\itemsep}{0pt}

\bibitem{BG} H. A. Biagioni and T. Gramchev, {\it Fractional derivative estimates in Gevrey classes, global regularity and decay for solutions to semilinear equations in $\RR^n$}, J. Differential Equations, {\bf 194} (2003), 140--165.

\bibitem{C1} M.~Cappiello, \textit{Gelfand spaces and pseudodifferential operators of infinite order in $\RR^n$}, Ann. Univ Ferrara, Sez. VII, Sc. Mat., \textbf{48} (2002), 75-97.
%
\bibitem{C2} M.~Cappiello, \textit{Pseudodifferential parametrices of infinite order for SG-hyperbolic problems}, Rend. Sem. Mat. Univ. Pol. Torino, \textbf{61} 4 (2003), 411-441.

\bibitem{CGR1} M. Cappiello, T. Gramchev and L. Rodino, {\it Super-exponential decay and holomorphic extensions for semilinear equations with polynomial coefficients}, J. Funct. Anal., {\bf 237} (2006), 634--654.

\bibitem{CGR3} M. Cappiello, T. Gramchev and L. Rodino, {\it Sub-exponential decay and uniform holomorphic extensions for semilinear pseudodifferential equations}, Comm. Partial Differential Equations, \textbf{35} (2010) n. 5, 846--877.

\bibitem{CGR4} M. Cappiello, T. Gramchev and L. Rodino, {\it Entire extensions and
 exponential decay for semilinear elliptic equations}, J. Anal. Math., \textbf{111} (2010), 339--367.

\bibitem{CN2} M. Cappiello and F. Nicola, {\it Regularity and decay of solutions of nonlinear harmonic oscillators}, Adv. Math. \textbf{229} (2012), 1266--1299.

\bibitem{CPP}
M. Cappiello, S. Pilipovi\'c and B. Prangoski, \textit{Parametrices and hypoellipticity for pseudodifferential operators on spaces of tempered ultradistributions}, J. Pseudo-Differ. Oper. Appl., online July 2014, http://dx.doi.org/10.1007/s11868-014-0095-3

\bibitem{CKK} J. Chung, S.Y. Chung and D. Kim, \textit{Characterizations of the Gelfand-Shilov spaces via Fourier transforms}, Proc. Am. Math. Soc. \textbf{124} 7 (1996), 2101-2108

\bibitem{faadib} G. M. Constantine and T. H. Savits, \textit{A multivariate Fa\`a di Bruno formula with applications}, Trans. Amer. Math. Soc., \textbf{348} 2 (1996), 503-520.

\bibitem{GR}
T. Gramchev, \textit{Perturbative methods in scales of Banach spaces: applications for Gevrey regularity of solutions to semilinear partial differential equations}, Rend. Sem. Mat. Univ. Pol. Torino, \textbf{61} (2003) no. 2, 101--134.

\bibitem{PilipovicK} R.~Carmichael, A.~Kami\'nski and S.~Pilipovi\'c, \textit{Boundary Values and Convolution in Ultradistribution Spaces}, World Scientific Publishing Co. Pte. Ltd., 2007


\bibitem{GS}
I.M. Gelfand and G.E. Shilov, \textit{Generalized functions II}, Academic Press, New York, 1968.

\bibitem{Komatsu1} H.~Komatsu, \textit{Ultradistributions, I: Structure theorems and a characterization}, J. Fac. Sci. Univ. Tokyo, Sect. IA Math., \textbf{20} 1 (1973), 25-105.

\bibitem{Komatsu2} H.~Komatsu, \textit{Ultradistributions, II: The kernel theorem and ultradistributions with support in submanifold}, J. Fac. Sci. Univ. Tokyo, Sect. IA Math., \textbf{24} 3 (1977), 607-628.

\bibitem{KomatsuI} H.~Komatsu, \textit{The implicit function theorem for ultradifferentiable mappings}, Proc. Japan Acad., Ser. A \textbf{55} 3 (1979), 69-72.

\bibitem{Komatsu3} H.~Komatsu, \textit{Ultradistributions, III: Vector valued ultradistributions and the theory of kernels}, J. Fac. Sci. Univ. Tokyo, Sect. IA Math., \textbf{29} 3 (1982), 653-717.


\bibitem{NR} F.~Nicola and L.~Rodino, \textit{Global Psedo-Differential Calculus on Euclidean Spaces}, Vol. 4. Birkh\" auser Basel, 2010

\bibitem{PilipovicU} S.~Pilipovic, \textit{Tempered ultradistributions}, Boll. Un. Mat. Ital. B (7) \textbf{2} (1988), no. 2, 235-251.

\bibitem{PilipovicT} S.~Pilipovi\'c, \textit{Characterizations of bounded sets in spaces of ultradistributions}, Proc. Amer. Math. Soc., \textbf{120}  (1994) no. 4, 1191--1206.


\bibitem{PP} S. Pilipovi\'c and B. Prangoski, \textit{Anti-Wick and Weyl quantization on ultradistribution spaces}, J. Math. Pures Appl., online April 2014, http://dx.doi.org/10.1016/j.matpur.2014.04.011


\bibitem{BojanL} B.~Prangoski, \textit{Laplace transform in spaces of ultradistributions}, Filomat \textbf{27} 5 (2013), 747-760.

\bibitem{BojanS} B.~Prangoski, \textit{Pseudodifferential operators of infinite order in spaces of tempered ultradistributions}, J. Pseudo-Differ. Oper. Appl. \textbf{4} (2013), 495-549.


\end{thebibliography}
\end{document}